\documentclass[a4paper, 10pt]{amsart}

\usepackage{amsmath}
\usepackage{amssymb}
\usepackage{amsfonts}
\usepackage{mathrsfs}
\usepackage{array}
\usepackage{enumerate}
\usepackage[french, english]{babel}
\usepackage[all]{xy}
\usepackage[T1]{fontenc}
\usepackage{float}

\numberwithin{equation}{section}

\newcommand{\msl}{\mathscr{L}}

    \newcommand{\BQ}{{\mathbb {Q}}} \newcommand{\BR}{{\mathbb {R}}}

     \newcommand{\BZ}{{\mathbb {Z}}}

    \newcommand{\CA}{{\mathcal {A}}} 
     
    \newcommand{\CE}{{\mathcal {E}}}

     \newcommand{\CL}{{\mathcal {L}}}
     
    \newcommand{\CO}{{\mathcal {O}}}

     \newcommand{\fd}{{\mathfrak{d}}}

     \newcommand{\ft}{{\mathfrak{t}}}

     \newcommand{\fA}{{\mathfrak{A}}} \newcommand{\fB}{{\mathfrak{B}}}
     \newcommand{\fD}{{\mathfrak{D}}}
    \newcommand{\fE}{{\mathfrak{E}}} 
    \newcommand{\fG}{{\mathfrak{G}}}

    \newcommand{\fS}{{\mathfrak{S}}} 
     
    \newcommand{\fW}{{\mathfrak{W}}}

 \renewcommand{\mod}{\ \mathrm{mod}\ }

    \theoremstyle{plain}
    \newtheorem{thm}{Theorem}[section] \newtheorem{cor}[thm]{Corollary}
    \newtheorem{lem}[thm]{Lemma}  \newtheorem{prop}[thm]{Proposition}

\newtheorem{exm}[thm]{Example}

\theoremstyle{remark} 
\theoremstyle{remark} 
\theoremstyle{remark} 

    \numberwithin{equation}{section}

\DeclareFontFamily{U}{wncy}{}
\DeclareFontShape{U}{wncy}{m}{n}{<->wncyr10}{}
\DeclareSymbolFont{mcy}{U}{wncy}{m}{n}

\begin{document}

\title[Quadratic Forms, $K$-groups and $L$-values]{Quadratic Forms, $K$-groups and $L$-values of elliptic curves}
\author{Li-Tong Deng, \, Yong-Xiong Li \, and \, Shuai Zhai}

\begin{abstract}
Let $f$ be a positive definite integral quadratic form in $d$ variables. In the present paper, we establish a direct link between the genus representation number of $f$ and the order of higher even $K$-groups of the ring of integers of real quadratic fields, provided $f$ is diagonal and $d \equiv 1 \mod 4$, by applying the Siegel mass formula. When $d=3$, we derive an explicit formula of $r_f(n)$ in terms of the class number of the corresponding imaginary quadratic field and the central algebraic values of $L$-functions of quadratic twists of elliptic curves, by exploring a theorem of Waldspurger. Moreover, by the $2$-divisibility results on the algebraic $L$-values of quadratic twist of elliptic curves, we obtain a lower bound for the $2$-adic valuation of $r_f(n)$ for some odd integer $n$. The numerical results show our lower bound is optimal for certain cases. We also apply our main result to the quadratic form $f=x_1^2+\cdots+x^2_d$ to determine the order of the higher $K$-groups numerically.

\end{abstract}

\subjclass[2010]{11E25 (primary),11R70 11G05 (secondary)}

\keywords{Quadratic forms, $K$-groups, $L$-functions, Elliptic curves}

\maketitle

\selectlanguage{english}

\section{Introduction}\label{int}

Let $f=f(x_1,...,x_d)$ be a positive definite integral quadratic form in $d \geq 2$ variables. Within the arithmetic theory of quadratic forms, a fundamental challenge is to determine the possible values that the form $f$ can take on when its variables $x_1, ..., x_d$ are integers. Furthermore, a key objective is to establish the number of representations of a positive integer $n$ by $f$, denoted as $r_f(n)$ and defined by the expression:
$$
r_f(n)=\#\{(x_1,...,x_d)\in \mathbb{Z}^d:f(x_1,...,x_d)=n\}.
$$
To discover a precise and relatively straightforward formula for $r_f(n)$ has consistently motivated the historical exploration of quadratic forms. In certain special cases, we can explicitly calculate $r_f(n)$, such as the well-known results of Gauss and Dirichlet when $d=2$. However, in general, providing a formula for a single $r_f(n)$ is challenging. Perhaps the most notable general result in this context is Siegel's well-known result, recognised as the Siegel mass formula, which links the genus representation number of $f$ to the $L$-value of certain quadratic field. While in this paper, we will explore the link between the genus representation number involving $r_f(n)$ and $K$-groups; and the connection among $r_f(n)$, $L$-values of elliptic curves and class numbers.

Throughout the paper, we assume that
$$
f(x_1,...,x_d)=b_1x_1^2+\cdots +b_dx_d^2
$$
is a diagonal quadratic form with $d$ odd satisfying $b_i\in\BZ$ $(1\leq i\leq d)$ and $\gcd (b_1,\cdots,b_d)=1$. The matrix associated to $f$ is defined by the diagonal matrix
\begin{equation}\label{1-1-f-1}
 A_f:=\left(\frac{\partial^2 f}{\partial x_i \partial x_j}\right)_{d\times d}=2\cdot{\rm diag}\{b_1,\cdots,b_d\}.
\end{equation}
Denote the group of automorphisms of $f$ by
$$
\mathcal{O}(f)=\{\alpha\in GL_d(\mathbb{Z}):\alpha^tA_f\alpha=A_f\}.
$$
It is clearly that $\mathcal{O}(f)$ is finite, and set $\mu_f=|\CO(f)|$. Here for a finite set $\fW$, we denote by $|\fW|$ the cardinality of $\fW$. Let $\mathfrak{G}_f$ be the genus of $f$. It is well known that $\mathfrak{G}_f$ is a finite set with cardinality $h_f$, which we also call the class number of $f$. For each $h \in \fG_f$, denote
\begin{equation}\label{1-1-f-2}
  \xi_h=\left(\mu_h\sum\limits_{g\in \mathfrak{G}_f}\frac{1}{\mu_g}\right)^{-1}.
\end{equation}
We then define the genus representation number of $f$ to be
\[G_f(n)=\sum\limits_{g\in \mathfrak{G}_f} \xi_g r_g(n).\]

For each prime $p$ (including the infinite prime $p=\infty$), we can view the equation
$$
f(x_1,\cdots,x_d)=n
$$
as a map from $\BZ^d_p$ to $\BZ_p$ (when $p=\infty$ we view the map from $\BR^d$ to $\BR$). Taking a Haar measure on $\BZ_p$, we obtain a Haar measure on $\BZ^d_p$. The density \(\delta_p(n,f)\) of solutions \(f(x)=n\) in \(\mathbb{Z}_p\) is defined to be
\begin{equation}\label{1-1-f-3}
\delta_p(n,f)=\lim_{U \to n} \frac{vol(f^{-1}(U))}{vol(U)},
\end{equation}
where \(U\) runs through the neighborhoods of \(n\in\mathbb{Z}_p\). One can show that when \(U\) is small enough, the ratio is stable (see \cite{Iwaniec}). Define \(\varepsilon_2=\frac{1}{2}\) and \(\varepsilon_d=1\) if \(d>2\). One has the following mass formula of Siegel.

\begin{thm}[Siegel, \cite{Iwaniec}]\label{S}
We have
\begin{equation}\label{SF}
G_f(n)=\varepsilon_d\prod\limits_{p\le\infty} \delta_p(n,f).
\end{equation}
\end{thm}

Siegel's formula gives a connection between coefficients of weighted average of theta series and the product of local representation densities of positive definite quadratic forms over $\mathbb{Z}_p$. The above formula is a special case of the Siegel--Weil formula, which establishes an identity between theta integrals and special values of Siegel Eisenstein series.We will apply the above formula to establish a direct link between the genus representation number $G_f$ and the order of higher even $K$-groups of the ring of integers of real quadratic fields.

\subsection{Main results on higher even $K$-groups}

Recall that $f(x_1,...,x_d)=b_1x_1^2+\cdots +b_dx_d^2$. Put $D=b_1\cdots b_d$. Assume that $d\geq5$ and $d\equiv 1\mod 4$. For any positive integer $n$ prime to $2D$, we define $F$ to be the real quadratic field $\mathbb{Q}(\sqrt{nD})$. Let $\CO_F$ be the ring of integers of $F$. For any positive even integer $m$, we define $K_{m}(\CO_F)$ to be the $m$-th $K$-group of $\CO_F$. By a fundamental result of Quillen \cite{Q} (see also \cite{Weibel}), we know $K_{m}(\CO_F)$ is finite.  Put $\fd=\frac{d-1}{2}$. Let $\zeta(s)$ be the Riemann zeta function and $w_\fd(F)=|H^0(F,(\BQ/\BZ)(\fd))|$. Define
	$$
		W_\fd(F)=
    		\left\{
    		\begin{aligned}
        		& w_\fd(F)    & \text{if } d \equiv 5 \mod 8; \\
		& \frac{w_\fd(F)}{4} & \text{if } d \equiv 1\mod 8.
    		\end{aligned}
    		\right.
	$$
The first main result of the paper is given by the following theorem.

\begin{thm}\label{main-thm-1}
Assume that $d\geq5$ and $d\equiv 1\mod 4$. Let $n$ be any positive square free integer which is coprime to $2D$. Then
we can find a rational number \(\rho_f(n)\) such that
\begin{equation}\label{main-thm-f-1}
G_f(n)=\frac{\rho_f(n)}{W_{\fd}(F)\zeta(1-\fd)} |K_{d-3}(\mathcal{O}_F)|.
\end{equation}
\end{thm}
\bigskip

The rational number $\rho_f(n)$ can be explicitly determined in \eqref{gamma2} and has finitely many possibilities (depending on $f$). In fact, in Corollary \ref{B2} we can find a finite set $\fS(d)$ of positive square free integers such that for any positive square free integer $n$,  $\rho_f(n)=\rho_f(s)$ for some $s\in\fS(d)$.
We remark $\rho_f(n)$ can be zero in the sense that $G_f(n)$ is also zero.

\medskip

\begin{exm}\label{exm-1-2}
Let \(d=5\), \(f=x_1^2+x_2^2+x_3^2+x_4^2+7x_5^2\), then \(h_f=2\) and the other element in $\fG_f$ is \(g=x_1^2+x_2^2+x_3^2+2x_4^2+4x_5^2+2x_4x_5\). In this case \(\fd=2\) and \(K=\mathbb{Q}(\sqrt{7n})\), which implies \(w_2(K)=24\). Theorem \ref{main-thm-1} shows that
\[r_f(n)+4r_g(n)=\frac{5\rho_f(n)}{-2} |K_2(\mathcal{O}_K)|,\]
where
	$$
		\rho_f(n)=
    		\left\{
    		\begin{aligned}
        		& -\frac{4}{5} & & n\equiv 1\mod 4 \\
        		& -\frac{28}{5} & & n\equiv 3\mod 8 \\
		& -\frac{12}{5} & & n\equiv 7\mod 8.
    		\end{aligned}
    		\right.
	$$
\end{exm}

\medskip

In view of the above theorem and example, one can determine the order of $K$-groups by studying the representation number of quadratic forms, and conversely, using the theory of $K$-groups to study quadratic forms. In the final section of the paper, we take
$$
f=x_1^2+x^2_2+\cdots+x^2_d.
$$
By using the formula \eqref{main-thm-f-1}, we numerically determine the order of $K_m(\CO_K)$ for $m=2,6,10$.

\medskip

Recall that $D=b_1\cdots b_d$. For a positive integer $b$, we denote by $h(-b)$ the class number of the imaginary quadratic field $\BQ(\sqrt{-b})$.
Using the same method in the proof of Theorem \ref{main-thm-1}, when $d=3$, we can obtain the following result relating $G_f(n)$ to the class number of certain imaginary quadratic fields.

\medskip

\begin{thm}\label{main-thm-2}
Let
\[f(x,y,z)=b_1 x^2+b_2 y^2+b_3 z^2\]
with $\gcd(b_1,b_2,b_3)=1$. For any positive square free integer $n$ coprime to $2D$, we can find a non-negative rational number $\lambda_f(n)$ such that
\[G_f(n)=\lambda_f(n)h(-Dn).\]
\end{thm}

\medskip

Theorem \ref{main-thm-2} directly follows from Proposition \ref{C} and Corollary \ref{B}. Here the number $\lambda_f(n)$ can be explicitly determined and has finitely many possibilities (depending on a similar set $\fS(3)$ as $\fS(d)$). The number $\lambda_f(n)$ may be zero, an example is given in Proposition \ref{gfp}. Analogous result also obtained in \cite{WZ}, but our methods are different.

\medskip

\subsection{Main results on quadratic twists of elliptic curves}
Recall that the theta series related to $f$ is defined by
\[\theta_f=\sum\limits_{x\in \mathbb{Z}^d} q^{f(x)}=\sum\limits_{n=0}^\infty r_f(n)q^n.\]
If \(g\) lies in the same genus of \(f\), then a classical result by Eichler \cite{Eichler} states that \(\theta_f-\theta_g\) is a cusp form of weight \(\frac{d}{2}\). Moreover, if \(d\) is odd and \(\theta_f-\theta_g\) is an eigenform, by Shimura lift, the coefficients of \(\theta_f-\theta_g\) is related to the coefficients of some modular forms of integral weight. We define the class number $h_f$ of $f$ to be the number of the genus of $f$.

We assume that $d=3$ and $h_f=2$. Thus we can write $\fG_f=\{f,g\}$. Recall that $A_f$ is the matrix associated to $f$, we define the level of $f$ to be the smallest positive integer $N=N_f$ such that $NA^{-1}_f$ has integral coefficients and even diagonal entries. It easily follows that $N$ is divisible by $4$ in view of \eqref{1-1-f-1}. Recall that $D=b_1b_2b_3$. By a theorem of Eichler \cite{Eichler}, $\theta_f-\theta_g$ is an eigen cusp form with weight $\frac{3}{2}$, level $N$ and character $\chi$ given by the Jacobian symbol $\left(\frac{4D}{\cdot}\right)$ (see \cite{Schulze}). Namely, $\theta_f-\theta_g\in S_{\frac{3}{2}}(N,\chi)$. Let us write
\[\theta_f-\theta_g=\sum^\infty_{n=1}a_nq^n. \]
By a fundamental theorem of Shimura \cite{Shimura} and a remark made by Niwa \cite{Niwa}, $\theta_f-\theta_g$ corresponds to a modular form of weight $2$, level $\frac{N}{2}$ and character $\chi^2$, we call the modular form the Shimura lift of $\theta_f-\theta_g$. By the modularity theorem due to Wiles et. al., we put the following assumption:
\begin{itemize}
  \item (\textbf{Sh-E}) The Shimura lift of $\theta_f-\theta_g$ corresponds to an elliptic curve $\CE$ over $\BQ$.
\end{itemize}

\medskip

For an elliptic curve $\CA$ over $\BQ$ of conductor $C$ and a square free integer $m$, we denote by $\CA^{(m)}$ the quadratic twist of $\CA$ by the extension $\BQ(\sqrt{m})$ over $\BQ$.
Let $L(\CA^{(m)},s)$ be the complex $L$-series of $\CA^{(m)}$. Let $\Omega_{\CA^{(m)}}=\Omega^+_{\CA^{(m)}}$ or $2\Omega^+_{\CA^{(m)}}$, according as $\CA^{(m)}(\BR)$ is connected or not. Here $\Omega^+_{\CA^{(m)}}$ is the minimal real period of a N\'{e}ron differential on a global minimal Weierstrass equation for $\CA^{(m)}$. We define the algebraic $L$-value of $\CA^{(m)}$ to be
\[\msl_{\CA}(m)=\frac{L(\CA^{(m)},1)}{\Omega_{\CA^{(m)}}}.\]
A well-known theorem of Birch shows that $\msl_{\CA}(m)$ is always a rational number. For each finite prime $\ell$, we denote by $c_\ell(\CA)$ the Tamagawa number of $\CA$ at $\ell$. For any square free integer $M\equiv 1\mod 4$ and $(M,C)=1$, we define
\[t_\CA(M)=\sum_{\ell\mid M}t(\ell),\]
where
$t(\ell)$ is equal to $1$ (respectively, $0$) according as $2$ divides (respectively, does not divide) $c_\ell(\CA^{(M)})$.

\medskip

To state our main result, we have to introduce some notations. For each prime $p$, let $\BQ_p$ be the field of $p$-adic numbers. For an integer $M$, we say that two square free positive integers $n_1,n_2$ are {\it $M$-square-linked} if $n_1/n_2$ is a square in $\BQ^\times_\ell$ for each prime $\ell$ dividing $M$. Obvious the  {\it $M$-square-linked} relation is an equivalent relation. For any integer $M\neq\pm1$, we can easily show that the set of equivalent classes of  {\it $M$-square-linked} relation is finite. Recall that $f=b_1x^2+b_2y^2+b_3z^2$, $D=b_1b_2b_3$ and $N_f$ is the level of $f$. We put $\fS_1$ to be a set of positive square free integers which consists of representatives of the equivalent classes of {\it $N_f$-square-linked}
relation.

Let $n>3$ be a square free integer prime to $2D$. Recall the elliptic curve $\CE$ defined via Shimura lift of $\theta_f-\theta_g$. Let $D^\circ$ be the square free part of $D$. We put $E=\CE^{(-D^\circ)}$ (respectively, $E=\CE^{(D^\circ)}$) if $n\equiv 1\mod 4$ (respectively, $n\equiv 3\mod 4$).

From Corollary \ref{B}, $\fS_1$ can be taken as a subset of $\fS(3)$.
Now we make the following assumption on the set $\fS_1$:
\begin{itemize}
  \item (\textbf{In-Nv}) There exists a positive square free integer $n_0\in\fS_1$ such that
  $$L(E^{(n_0)},1)\neq0.$$
\end{itemize}
In fact, by the main theorem in \cite{BFH}, when the sign of $E$ is $+1$, one can always find such a $n_0$ satisfying (\textbf{In-Nv}).

From now on, we simply write $\CL(m)$ for $\msl_{\CE}(-D^\circ m)$ for any square free integer $m$. Recall that $\fG_f=\{f,g\}$. From definition \eqref{1-1-f-2}, we have
\[\xi_g=\frac{\mu_f}{\mu_f+\mu_g}\quad\textrm{ and }\quad \xi_f=1-\xi_g,\]
we then simply write $\xi$ for $\xi_g$. Denote by $v_2$ the additive normalised $2$-adic valuation on $\overline{\BQ}_2$ such that $v_2(2)=1$. Recall the rational number $\lambda_f(n)$ in Proposition \ref{main-thm-2}.
We define $\fA$ to be the minimal number of all $v_2\left(\frac{\xi a_{n_0}}{\sqrt{\CL(n_0)}}\right)$ where $n_0$ runs over all elements in $\fS_1$ satisfying (\textbf{In-Nv}). We also define $\fB$ to be the minimal number of all $v_2(\lambda_f(s))$, where $s$ runs over all elements in $\fS(3)$. Let $\mu(m)$ be the number of distinct odd prime divisors of $m$.

\medskip

Our second main result is given by the following theorem.

\begin{thm}\label{main-thm-3}
Assume that $d=3$ and $h_f=2$, say $\fG_f=\{f,g\}$.
Let $n>3$ be a square free integer coprime to $2D$. Assume that
\[\theta_f-\theta_g=\sum^\infty_{n=1}a_n q^n\]
and
the condition (\textbf{Sh-E}) holds.
\begin{enumerate}
  \item [(1)]We have $r_f(n)=\xi a_n+\lambda_f(n)\cdot h(-Dn)$.
  \item [(2)]If the condition (\textbf{In-Nv}) holds and the two integers $n$ and $n_0$ are {\it $N_f$-square-linked}, then
  \[r_f(n)=\pm\xi a_{n_0}\sqrt{\frac{\CL(n)}{\CL(n_0)}}+\lambda_f(n)h(-Dn).\]
  \item [(3)]Suppose that (\textbf{In-Nv}) holds and the two integers $n$ and $n_0$ are {\it $N_f$-square-linked}. Assume that $E$ is an optimal elliptic curve with conductor $C_E$ and Manin constant $\nu_E$. We put
\[\kappa=\max\{2(\fA-\fB-\mu(D))+1-v_2(\nu_E),1\}.\]
      Then
   \[r_f(n)\geq \fA+\frac{1}{2}(t_E(n^*)-1-v_2(\nu_E))\]
   holds for all $n$ such that $\mu(n)\geq\kappa$. Here $n^*=\left(\frac{-1}{n}\right)n$.
\end{enumerate}
\end{thm}
We remark in (2) of the above theorem, only one sign appears in the equality. However, at present, we see no way how to determine the precise sign.
Similar results as (3) in the Theorem \ref{main-thm-3} were obtained only for special elliptic curves (e.g. $y^2=x^3-x$) in \cite{WZ} (see also \cite{Qin} for other applications for the same curve).

\medskip

\begin{exm}\label{E7}
Consider the special case \(f=x_1^2+x_2^2+7x_3^2\). Then \(g=x_1^2+2x_2^2+4x_3^2+2x_2x_3\), \(E=X_0(14)\), and \[r_f(n)+2r_g(n)=3\lambda_f(n)h(-7n),\]
with
$$
		\lambda_f(n)=
    		\left\{
    		\begin{aligned}
        		\frac{2}{3} & \quad n\equiv 3\ (mod\ 4)\\
			\frac{8}{3} & \quad n\equiv 1\ (mod\ 8)\\
        		4 & \quad n\equiv 5\ (mod\ 8).\\
    		\end{aligned}
    		\right.
	$$
Our Theorem \ref{main-thm-3} shows for any square free integer \(n\) coprime to \(14\),
\[v_2(r_f(n))\ge \frac{3}{2}+\frac{1}{2}\mu(n).\]
When $n=3705$, numerical computation in Section \S\ref{num-1} shows the lower bound can be obtained.
\end{exm}

\medskip

\subsection{Idea for the proofs.}

The key ingredient in the proof of Theorem \ref{main-thm-1} and Proposition \ref{main-thm-2} is Siegel's mass formula, which says the genus representation number $G_f(n)$ is a product of the local representation densities over all primes. We then make a very delicate study on the diagonal equations over finite fields, based on this, we determine the local representation densities up to a finite set of primes. As a result, the number $G_f(n)$ is related to the Dirichlet $L$-values at non-positive integers. Since the Iwasawa main conjecture holds for abelian extensions over $\BQ$ for all primes, we know the Quillen--Lichtenbaum conjecture holds. Thus Theorem \ref{main-thm-1} follows.
\medskip

For the proof of Theorem \ref{main-thm-3}, besides the above Siegel's formula, we also need two more key ingredients. The first is a special value formula of Waldspurger. Under the condition (\textbf{Sh-E}), we know $\theta_f-\theta_g=\sum_n a_n q^n$ corresponds to elliptic curves over $\BQ$ via Shimura lift. Then Waldspurger formula gives a precise relation between the $L$-values of elliptic curves with the coefficients of $\theta_f-\theta_g$. Thus, under the condition (\textbf{In-Nv}), we can relate $r_f(n)$ to the algebraic $L$-values of quadratic twist of elliptic curves over $\BQ$. Here comes our second key input: by an induction argument by using modular symbols, one of us \cite{Zhai} proved a very general lower bound for the $2$-adic valuations for the algebraic $L$-values of the quadratic twist of optimal elliptic curves over $\BQ$. Based on this bound, we get a $2$-adic lower bound for $r_f(n)$ and prove Theorem \ref{main-thm-3}.

\medskip

In the final part of the introduction, we give the structure of the paper. In the next section \S\ref{PL}, we determinate the local representation densities in the mass formula of Siegel. In section \S\ref{siegel-chp}, we use Siegel's formula to prove Theorem \ref{main-thm-1} and Proposition \ref{main-thm-2}. In section \S\ref{walds}, we give the definition of Shimura lift and the formula of Waldspurger, then we prove Theorem \ref{main-thm-3} using the $2$-adic divisibility result. In section \S\ref{num-1}, we give several numerical examples related to Theorem \ref{main-thm-3}, in particular, some numerical data shows our lower bound is optimal. In the final section \S\ref{ND}, we apply the formula in Theorem \ref{main-thm-1} to the quadratic form $f=x_1^2+\cdots+x^2_d$ with $d=5,9,13$, we then get numerical data for the $K$-groups $K_m(\CO_K)$ for $m=2,6,10$.

\bigskip

\section{Local representation densities}\label{PL}

In this section, we determinate the local representation densities in the mass formula of Siegel. We now fix a diagonal quadratic form
$$
f=b_1x_1^2+\cdots +b_dx_d^2.
$$
Besides Lemma \ref{delta}, we always assume $d$ is an odd positive integer. When $d=3$, we may also write $f$ as $f=b_1x^2+b_2y^2+b_3z^2$. We set $D=b_1b_2\cdots b_d$. We remark that the level $N_f$ is always divisible by $4$. Recall that the density \(\delta_p(n,f)\) of solutions \(f(x)=n\) in \(\mathbb{Z}_p\) is defined to be
$$
\delta_p(n,f)=\lim_{U \to n} \frac{vol(f^{-1}(U))}{vol(U)},
$$
where \(U\) runs through the neighborhoods of \(n\in\mathbb{Z}_p\).

We first deal with the finite prime $p$, in this case,
there exists an integer \(\beta_p(n,f)\) such that for any integer \(\beta\ge \beta_p(n,f)\),
\begin{equation}\label{delta-f-1}
  \delta_p(n,f)=p^{-\beta(d-1)}\cdot|\{x\in (\mathbb{Z}/p^\beta\BZ)^d:f(x)\equiv n\ (mod\ p^\beta)\}|.
\end{equation}
We first determine \(\beta_p(n,f)\) using the following version of Hensel's Lemma.

\begin{lem}[Hensel]\label{H}
Let \(F(x)\) be a polynomial with integer coefficients. If \(r\) is an integer such that \(F(r)\equiv 0\ (mod\ p^k)\) and \(F(r)\not\equiv 0\ (mod\ p^{k+1})\) for some positive integer \(k\), and \(F'(r)\equiv 0\ (mod\ p^\ell)\), \(F'(r)\not\equiv 0\ (mod\ p^{\ell+1})\) for some integer \(\ell<\frac{k}{2}\), then for every integer \(\tau>k-\ell\), there is a unique integer \(s_\tau\) such that
$f(s_\tau)\equiv 0\ (mod\ p^\tau)$ with $s_\tau\equiv r\ (mod\ p^{k-\ell})$.
\end{lem}

\medskip

\begin{lem}\label{beta-p}
Let $n$ be a square free positive integer which is coprime to $2D$. Then we can take
\[\beta_p(n,f)=
\begin{cases}
1,\quad &\textrm{if $p\nmid 2nD$}\\
2,\quad &\textrm{if $p\mid n$.}
\end{cases}\]
Assume that $d=3$, $\gcd(b_1,b_2,b_3)=1$ and $n$ is coprime to $2D$, we can take $\beta_2(n,f)=3$.
\end{lem}

\begin{proof}
For the first assertion, we only prove the case when $p\nmid 2nD$, the other case is similar. Note the equation
\begin{equation}\label{beta-p-f-1}
  b_1x_1^2+\cdots+b_dx^2_d\equiv n\mod p.
\end{equation}
Under our assumption $p\nmid 2Dn$, if there is a $b_i$, say $b_1$, which is a quadratic residue modulo $p$,
by taking $x_2=\cdots=x_d=0$ and $F(x_1)=b_1x^2_1-n$ in Lemma \ref{H}, we can let $k=1,\ell=0$ in Lemma \ref{H}; if all $b_i$'s are quadratic non residue modulo $p$, via considering an equivalence equation to \eqref{beta-p-f-1}
\[x^2_1+\frac{b_2}{b_1}x_2^2+\cdots+\frac{b_d}{b_1}x^2_d\equiv \frac{n}{b_1}\mod p,\]
we reduce to the previous case, thus we can also let $k=1,\ell=0$ in Lemma \ref{H}. Now we show that $\beta_p(n,f)=1$. In fact,  for integer \(\beta>1\) and $x_i\in\BZ(1\leq i\leq d)$, we
let
$$x_i=y_i+p^{\beta-1}z_i,\quad n=n_0+p^{\beta-1}n_1,$$
where $n_1, z_i(1\leq i\leq d)$ are integers and $y_i,n_0\in \BZ\cap[0,p^{\beta-1}-1]$.
Then the equation
$$b_1x_1^2+...+b_dx_d^2\equiv n\ mod\ p^{\beta}$$
implies
\begin{equation}\label{beta-p-f-2}
  \left\{
    		\begin{aligned}
        		& b_1y_1^2+..+cy_d^2\equiv n_0\ (mod\ p^{\beta-1})\\
			& 2b_1y_1z_1+...+2b_dy_dz_d\equiv \frac{n_0-(b_1y_1^2+..+cy_d^2)}{p^{\beta-1}}+n_1\ (mod\ p)
    		\end{aligned}
    		\right.
\end{equation}
Note that \(n\) is square free, the \(y_i\)'s cannot be divisible by $p$ at the same time. Thus for a fixed non-zero root \((y_1,...,y_d)\) of \(b_1y_1^2+...+b_dy_d^2\equiv n_0\ (mod\ p^{\beta-1})\), we can lift it to \((x_1,...,x_d)\) via the second congruence in \eqref{beta-p-f-2}, and there are \(p^{d-1}\) of such lift.
Thus for $\beta>1$, the number in the right hand of the equality \eqref{delta-f-1} is stable, which must be $\delta_p(f,n)$. Therefore
\(\beta_p(f)=1\).

When $d=3$, by our assumption, we can take $\ell=1,k=3$ in Lemma \ref{H}. By the following modification to the method on the lifting solutions given above,
i.e., for $\beta>1$, we let
\[n=n_0+4p^{\beta-1}n_1\quad x_i=y_i+2p^{\beta-1}z_i,\quad 1\leq i\leq 3.\]
Here, all $n_0,n_1,y_i,z_i$ are integers with $n_0\in\BZ\cap[0,2^{\beta+1}-1]$ and $y_i\in\BZ\cap[0,2^\beta]$. We consider the congruences via modulo $2^{\beta+2}$ and $2^{\beta+1}$, and the second congruence in \eqref{beta-p-f-2} becomes
\[b_1y_1z_1+b_2y_2z_2+b_3y_3z_3\equiv n_1+\frac{n_0-(b_1y^2_1+b_2y^2_2+b_3y^2_3)}{2^{\beta+1}}\mod 2.\]
Then by the same method as above
we can take $\beta_p(f,n)=3$.

\end{proof}

To determine the local densities $\delta_p(f,n)$, we need a simple result on the sum of quadratic characters.

\begin{lem}\label{q-sum}
Let \(p\) be an odd prime, then
	$$
		\sum\limits_{x=1}^p \left(\frac{x^2+a}{p}\right)=
    		\begin{cases}
        		 -1 &\textrm{if $p\nmid a$}\\
			 p-1 &\textrm{if  $p|a$.}
    		\end{cases}
    		$$
\end{lem}

\begin{proof}
Note that \(y^2-x^2\equiv a\ (mod\ p)\) has \(p-1\) solutions \((x,y)\) if \(p\nmid a\) and has \(2p-1\) solutions if \(p|a\). Thus the result follows from that the cardinality of the set
$\{x\mod p:x^2\equiv a\mod p\}$ is $1+\left(\frac{a}{p}\right)$.
\end{proof}

Now we can determine the $\delta_p(f,n)$ in the following lemmas. As remarked before, this lemma is the only place in the paper where we may assume $d$ is even.
\begin{lem}\label{delta}
Recall that \(D=b_1...b_d\). Let \(n\) be a square free positive integer prime to \(2D\), then for all primes \(p\nmid 2nD\), we have
	$$\delta_p(n,f)=
    		\begin{cases}
        		 1-p^{-\frac{d}{2}}\left(\frac{(-1)^{\frac{d}{2}}D}{p}\right) & \textrm{if $2|d$,}\\
        		 1+p^{\frac{1-d}{2}}\left(\frac{(-1)^{\frac{d-1}{2}}nD}{p}\right) & \textrm{if $2\nmid d$.}
    		\end{cases}$$
\end{lem}

\begin{proof}
By Lemma \ref{beta-p}, we know \(\beta_p(n,f)=1\). We only prove the case for $d$ odd, the other case can be done in the same way. The main idea is by induction on $d$.
For \(d=1\), since $p$ is coprime to $D=b_1$, we have
\[\delta_p(n,f)= |\{x\in \mathbb{Z}/p\BZ: x^2\equiv b_1n \mod p)\}|
=1+\left(\frac{b_1n}{p}\right).\]
We assume the lemma holds for \(1,3,...,d-2\). For \(d\), recall that
\[\delta_p(n,f)=p^{1-d}\cdot |\{x\in (\mathbb{Z}/p\BZ)^d: f(x)\equiv n \mod p)\}|.\]
Since $p$ is coprime to $2D$, the congruence
\[b_1x^2_1+\cdots+b_dx^2_d\equiv n\mod p\]
is equivalent to
\[x^2_1\equiv \frac{n}{b_1}-\frac{b_2}{b_1}x^2_2-\cdots-\frac{b_d}{b_1}x^2_d \mod p.\]
Set $\fE=\frac{n}{b_1}-\frac{b_2}{b_1}x^2_2-\cdots-\frac{b_d}{b_1}x^2_d$. Note that the number of solutions $x^2_1\equiv \fE\mod p$ for any $(x_2,\cdots,x_d)\in(\BZ/p\BZ)^{d-1}$ is equal to $1+\left(\frac{\fE}{p}\right)$. Thus we obtain
\begin{equation}\label{delta-p-f-1}
 \begin{aligned}\delta_p(n,f)=&p^{1-d}\sum^p_{x_2,\cdots,x_d=1}\left(1+\left(\frac{\fE}{p}\right)\right)\\
=&1+p^{1-d}\cdot\left(\frac{-b_1b_2}{p}\right)\sum^p_{x_2,\cdots,x_d=1}\left(\frac{x^2_2+\fD}{p}\right).\end{aligned}
\end{equation}
Here we denote $\fD=\frac{b_3}{b_2}x^2_3+\cdots+\frac{b_d}{b_2}x^2_d-\frac{n}{b_2}$ and use the relation $\fE=-\frac{b_2}{b_1}(x^2_2+\fD)$.
Now we apply Lemma \ref{q-sum} to the last sum in \eqref{delta-p-f-1}, we get
\begin{equation}\label{delta-p-f-2}
\sum^p_{x_2,\cdots,x_d=1}\left(\frac{x^2_2+\fD}{p}\right)=\sum^p_{x_3,\cdots,x_d=1}\sum^p_{x_2=1}\left(\frac{x^2_2+\fD}{p}\right)
=-p^{d-2}+\sum^p_{x_3,\cdots,x_d=1, p\mid \fD}p.
\end{equation}
Equivalently, we have
\[\begin{aligned}
&\sum^p_{x_2,\cdots,x_d=1}\left(\frac{x^2_2+\fD}{p}\right)+p^{d-2}\\
=&p\cdot|\{(x_3,\cdots,x_d)\in(\BZ/p\BZ)^{d-2}:b_3x^2_3+\cdots+b_dx^2_d\equiv n \mod p\}|.
\end{aligned}\]
Now applying the induction hypothesis to the above set
and put the above formula back into \eqref{delta-p-f-2}, we obtain the required result for $\delta_p(n,f)$.
\end{proof}

\begin{lem}\label{delta2}
Assume \(d\) is odd and $n$ is a square free positive integer with \((n,2D)=1\). For any odd prime \(p\) satisfying \(p|n\), we have
\[\delta_p(n,f)=1-\frac{1}{p^{d-1}}\]
\end{lem}

\begin{proof}
By Lemma \ref{beta-p}, we know $\beta_p(n,f)=2$. We consider the congruence
\begin{equation}\label{delta2-f-2}
  b_1x^2_1+\cdots+b_d x^2_d\equiv 0\mod p^2.
\end{equation}
Here we note that $p\mid n$. Write $x_i=y_i+pz_i(1\leq i\leq d)$ with $y_i,z_i\in\BZ$ and $y_i\in\BZ\cap[0,p-1]$. By the same lifting method as in Lemma \ref{beta-p}, we know the solution of the congruence \eqref{delta2-f-2} is equivalent to the following two congruences
\begin{equation}\label{delta2-f-1}
  \left\{
    		\begin{aligned}
        		& b_1y_1^2+..+b_dy_d^2\equiv 0\ (mod\ p)\\
			& 2b_1y_1z_1+...+2b_dy_dz_d\equiv \frac{-(b_1y_1^2+..+cy_d^2)}{p} \mod p
    		\end{aligned}
    		\right.
\end{equation}
Note that $n$ is square free, so all the $y_i(1\leq i\leq d)$'s can not be divisible by $p$.
Thus the number of solutions in $(\BZ/p^2\BZ)^d$ of the congruence \eqref{delta2-f-2} is $p^{d-1}$ times the number of non-zero solutions in $(\BZ/p\BZ)^d$ of the congruence
\[b_1y_1^2+..+b_dy_d^2\equiv 0 \mod p.\]
Now applying the same method in Lemma \ref{delta}, we have
\[|\{(y_1,\cdots,y_d)\in(\BZ/p\BZ)^d:b_1y_1^2+..+b_dy_d^2\equiv 0 \mod p\}|=p^{d-1}.\]
Therefore we get
\[|\{(x_1,\cdots,x_d)\in(\BZ/p^2\BZ)^d:b_1x_1^2+..+b_dx_d^2\equiv 0 \mod p^2)\}|=p^{d-1}(p^{d-1}-1),\]
and $\delta_p(n,f)=1-\frac{1}{p^{d-1}}$.

\end{proof}

It remains to consider the infinite prime case, i.e., $p=\infty$. In this case, we use the usual Lebesgue measure.
Recall that $D=b_1b_2\cdots b_d$. Write $\Gamma(s)$ for the Gamma function.

\begin{lem}\label{delta-inf}
For \(p=\infty\), \(\delta_\infty(n,f)=\frac{\pi^{\frac{d}{2}} n^{\frac{d}{2}-1}}{\Gamma(\frac{d}{2})\sqrt{D}}\).
\end{lem}

\begin{proof}
Let \(U\) be open interval in $\BR$ centered at \(n\) with radius \(R<n\).
Then
\[\begin{aligned}
vol(f^{-1}(U))=&\int_{|b_1x_1^2+...+b_dx_d^2-n|<R} dx_1\cdots dx_d\\
=&\frac{1}{\sqrt{D}}\int_{|x_1^2+...+x_d^2-n|<R} dx_1\cdots dx_d\\
=&\frac{\pi^{\frac{d}{2}}}{\Gamma(\frac{d}{2}+1)\sqrt{D}}((n+R)^{\frac{d}{2}}-(n-R)^{\frac{d}{2}}).
\end{aligned}\]
Here the second equality follows by the change of variable formula and the third equality follows from the volume of ball with radius $R$ in $\BR^d$ is equal to
\[\frac{\pi^{\frac{d}{2}}}{\Gamma\left(\frac{d}{2}+1\right)}R^d.\]
Note that $vol(U)=2R$ and $\Gamma(s+1)=s\Gamma(s)$, thus we have
\[\delta_\infty(n,f)=\frac{\pi^{\frac{d}{2}}}{\Gamma(\frac{d}{2}+1)\sqrt{D}} \lim\limits_{R\to 0} \frac{(n+R)^{\frac{d}{2}}-(n-R)^{\frac{d}{2}}}{2R}=\frac{\pi^{\frac{d}{2}} n^{\frac{d}{2}-1}}{\Gamma(\frac{d}{2})\sqrt{D}}.\]
\end{proof}

\medskip

We will now clarify the case $d=3$, and we can calculate \(\delta_p(n,f)\) for all \(p\le \infty\) explicitly. In this case $D=b_1b_2b_3$ and $f=b_1x^2+b_2y^2+b_3z^2$.

\begin{cor}\label{delta3}
When \(d=3\) and \(f=b_1x^2+b_2y^2+b_3z^2\) with \((b_1,b_2,b_3)=1\). Let $n$ be a square free positive integer satisfying \((n,2D)=1\). The local representation density \(\delta_p(n,f)\) has the following explicit form:
\begin{enumerate}
   \item For \(p=\infty\), \(\delta_\infty(n,f)=\frac{2\pi \sqrt{n}}{\sqrt{D}}\).
  \item For \(2<p<\infty\) and \(p\nmid n\), we can take \(\beta_p(f)=1\) and
	$$
		\delta_p(n,f)=\begin{cases} 1+\frac{1}{p}\left(\frac{-nD}{p}\right) &\textrm{if  $p\nmid D$,}\\
        		 1-\frac{1}{p}\left(\frac{-b_1b_2}{p}\right) &\textrm{if $p\nmid b_1b_2$ and $p|b_3$,}\\
			 1+\left(\frac{nb_1}{p}\right) &\textrm{if $p\nmid b_1$ and $p|b_2,p|b_3$.}
    		\end{cases}
	$$
  \item For \(2<p<\infty\) and \(p|n\), we can take \(\beta_p(f)=2\) and
\[\delta_p(n,f)=1-\frac{1}{p^2}.\]

  \item For \(p=2\), we can take \(\beta_p(f)=3\) and hence \(\delta_2(n,f)\) is completely determined by \(b_1,b_2,b_3,n\ (mod\ 8)\).
	
\end{enumerate}
\end{cor}

\begin{proof}
Assertions (1), (2) and (3) directly follow from Lemmas \ref{delta-inf}, \ref{delta}, \ref{delta2} and \ref{beta-p}. For (4), by Lemma \ref{beta-p}, we can take \(\beta_p(f)=3\), and hence the value of \(\delta_2(n,f)\) depends only on \(b_1,b_2,b_3,n\ (mod\ 8)\).
\end{proof}

\bigskip

\section{Siegel's mass formula and applications}\label{siegel-chp}

In this section, we use Siegel's mass formula to prove Theorem \ref{main-thm-1} and Proposition \ref{main-thm-2}.

\subsection{Siegel's Mass formula}
We shall fix the quadratic form $f=b_1x^2_1+\cdots+b_d x^2_d$.
Define \(\varepsilon_2=\frac{1}{2}\) and \(\varepsilon_d=1\) if \(d>2\). We have the following mass formula of Siegel:
$$
G_f(n)=\varepsilon_d\prod\limits_{p\le\infty} \delta_p(n,f)
$$

\subsection{Application to the case $d=3$.}

In view of the above Siegel's formula, the product equals to certain Fourier coefficients of a normalized Eisenstein series, which is related to $L$-values. When \(d=3\) we can express the product as class number \(h(-b)\) of an imaginary quadratic field \(\mathbb{Q}(\sqrt{-b})\) for some \(b>0\).

\begin{prop}\label{C}
Let \(f(x,y,z)=b_1x^2+b_2y^2+b_3z^2\) with \(\gcd(b_1,b_2,b_3)=1\) and $D=b_1b_2b_3$. Then for any positive square free integer \(n\ge 1\) coprime to \(2D\), we can find a non-negative rational number \(\lambda_f(n)\) such that
\begin{equation}\label{CF}
G_f(n)=\lambda_f(n)h(-Dn),
\end{equation}
where the number $\lambda_f(n)$ is given as
\begin{equation}\label{gamma}
\lambda_f(n)=\frac{24}{w_K}\sqrt{\frac{n}{D|D_K|}} \prod_{p|2D} \delta_p(n,f) \left(1-\frac{\chi(p)}{p}\right)\left(1-\frac{1}{p^2}\right)^{-1}.
\end{equation}
Here $K=\BQ(\sqrt{-Dn})$, $D_K$ is the discriminant of $K$, $w_K$ is the number of roots of unity in $K$, and $\chi$ is the Dirichlet character associated to $K/\BQ$.
\end{prop}

Note that the square root part in $\lambda_f(n)$ is independent of $n$. Let $D^\circ$ denote the square free part of $D$. Thus by $D/D^\circ$ is a square in $\BQ$, we know $\lambda_f(n)$ is indeed a rational number.

\begin{proof}
The key idea is that the local representation densities are essential local Euler factors of $L$-functions.
From Corollary \ref{delta3}, for all \(p\nmid 2Dn\), we have
\[\delta_p(n,f)=1+\frac{\chi(p)}{p}=\left(1-\frac{1}{p^2}\right)\left(1-\frac{\chi(p)}{p}\right)^{-1}.\]
From this relation and Siegel's formula in Theorem \ref{S}, we have
\[\begin{aligned}
G_f(n)=&\prod\limits_{p\le\infty} \delta_p(n,f)\\
=&\delta_\infty(n,f)\frac{L(1,\chi)}{\zeta(2)}\times \prod_{p|2Dn} \left( \delta_p(n,f) (1-\frac{\chi(p)}{p})(1-\frac{1}{p^2})^{-1}\right)\\
=&\frac{2\pi\sqrt{n}}{\sqrt{D}} \frac{6}{\pi^2}\frac{2\pi h(-nD)}{w_K\sqrt{|D_K|}} \prod_{p|2Dn} \delta_p(n,f) (1-\frac{\chi(p)}{p})(1-\frac{1}{p^2})^{-1},
\end{aligned}\]
where we have used the following standard facts on a zeta value and the class number formula
\[\zeta(2)=\frac{\pi^2}{6}\quad\textrm{and}\quad L(1,\chi)=\frac{2\pi h(-nD)}{w_K\sqrt{|D_K|}}.\]
Recall (3) in Corollary \ref{delta3}, we know for any $p\mid n$, $\delta_p(n,f)=1-\frac{1}{p^2}$. Therefore, we get
\[G_f(n)=\lambda_f(n)h(-nD)\]
and the number $\lambda_f(n)$ is as given in \eqref{gamma}.
\end{proof}

For an integer $M$, We recall the  {\it $M$-square-linked} relation defined in the introduction.
Two square free positive integers $n_1,n_2$ are {\it $M$-square-linked} if $n_1/n_2$ is a square in $\BQ^\times_\ell$ for each prime $\ell$ dividing $M$.

\begin{cor}\label{B}
There is a finite set $\fS(3)$ of positive square free integers such that
\[\lambda_f(n) \in \{\lambda_f(s)\}_{s\in \fS(3)}\]
for all positive square free integer \(n\) coprime to \(2D\) .
\end{cor}

\begin{proof}
Put $1,3$ in $\fS(3)$, we may assume \(n>3\).
Recall the formula in (\ref{gamma}), we know
\begin{enumerate}
\item \(|D_K|\) depends on \(n\ (mod\ 4)\).

\item \(\delta_2(n,f)\) depends on \(n\ (mod\ 8)\) by Corollary \ref{delta3}.

\item For any odd prime \(p|D\), $\delta_p(n,f)$ and $\chi(p)$
depend on \(n\ (mod\ 4D)\).
\end{enumerate}
Thus the cardinality of $\fS(3)$ can be bounded by the number of the finite set
$$\BZ/8\BZ\times \BZ/4D\BZ\times \{1,3\}.$$
Notice the dependence of $\lambda_f(n)$ on $n$ in (1), (2) and (3) are mainly related to quadratic residues. Thus
if an integer $n$ is a square in all $\BQ_p$ for $p\mid 2D$, from formula \eqref{gamma} and Corollary \ref{delta3}, we know $\lambda_f(n)$ is independent of $n$.
Therefore we can let $\fS(3)$ be a finite set with cardinality no larger than $2+8D$ containing a set of representatives of {\it $2D$-square-linked} equivalence classes.
\end{proof}

\subsection{Application to the case $d\geq5$}

We now assume \(d\ge 5\) to be an odd integer. Assume that $f=b_1x^2_1+\cdots+b_dx^2_d$ satisfies $\gcd(b_1,\cdots,b_d)=1$.
Recall that $D=b_1b_2\cdots b_d$ and $D^\circ$ the square free part of $D$. Put $\fd=\frac{d-1}{2}$ and let \(K=\mathbb{Q}(\sqrt{(-1)^{\fd}Dn})\). Define \(\chi\) to be the Dirichlet character associated to $K/\BQ$. Let $L(s,\chi)$ be the Dirichlet $L$-series associated to $\chi$, since $\chi$ is non-principal, the series $L(s,\chi)$ can be extended to a holomorphic function on the whole complex plane. We denote by $\varepsilon_K\in\{1,2\}$
satisfying the discriminant \(D_K=\varepsilon_K^2(-1)^{\fd}nD^\circ\).

\begin{prop}\label{K}
For any positive square free integer \(n\) which is coprime to \(2D\), there exists a rational number \(\rho_f(n)\) such that
\begin{equation}\label{KF}
G_f(n)=\rho_f(n)L\left(1-\fd,\chi\right).
\end{equation}
Moreover, the rational number $\rho_f(n)$ is given as
\begin{equation}\label{gamma2}
\rho_f(n)=\frac{2^{\fd-1}(-1)^{[\frac{\fd}{2}]}\sqrt{\frac{D^\circ}{D}}}{(D^\circ)^\fd\varepsilon_K^{2\fd-1}\zeta(1-2\fd)}\cdot\prod_{p|2D} \delta_p(n,f) \left(1-\frac{\chi(p)}{p^\fd}\right)\left(1-\frac{1}{p^{2\fd}}\right)^{-1}.
\end{equation}
Here we denote by $[\ft]$ the maximal integer no greater than $\ft$.
\end{prop}

We remark since $D^\circ$ is the square free part of $D$, thus $\sqrt{D^\circ/D}$ is a rational number.

\begin{proof}
The main idea is similar as Proposition \ref{C}, i.e., the local representation densities provide the local Euler factor for $L$-functions. In contrast to the case $d=3$, we make essentially use of the functional equations of both the Riemann zeta function and Dirichlet $L$-function.

From Lemmas \ref{delta} and \ref{delta2}, we know that
\[\delta_p(n,f)=
    		\begin{cases}
        		 (1-p^{-2\fd})(1-p^{-\fd}\chi(p))^{-1} & \textrm{if $p\nmid2nD$,} \\
        		(1-p^{-2\fd}) & \textrm{if $p|n,p\nmid 2D$,}
    		\end{cases}\]
and \(\delta_\infty(n,f)=\frac{\pi^{\frac{d}{2}} n^{\frac{d}{2}-1}}{\Gamma(\frac{d}{2})\sqrt{D}}\).

By Siegel's formula in Theorem \ref{S}, we have
\begin{equation}\label{K-f-1}
  G_f(n)
=\delta_\infty(n,f)\frac{L(\fd,\chi)}{\zeta(2\fd)}\cdot \prod_{p|2Dn} \delta_p(n,f) \left(1-\frac{\chi(p)}{p^\fd}\right)\left(1-\frac{1}{p^{2\fd}}\right)^{-1}.
\end{equation}
Recalling that
\[\delta_\infty(n,f)=\frac{\pi^{\fd+\frac{1}{2}}n^{\fd-\frac{1}{2}}}{\Gamma(\fd+\frac{1}{2})\cdot\sqrt{D}},\]
and applying the functional equations of both zeta and $L$-functions, we get
\begin{equation}\label{K-f-2}
 \delta_\infty(n,f)\frac{L(\fd,\chi)}{\zeta(2\fd)}=\frac{2^{\fd-1}(-1)^{[\frac{\fd}{2}]}
 \cdot\sqrt{\frac{D}{D^\circ}}}{(D^\circ)^\fd\varepsilon^{2\fd-1}_K}\cdot\frac{L(1-\fd,\chi)}{\zeta(1-2\fd)}.
\end{equation}
By Lemma \ref{delta2}, for any prime $p\mid n$ but $p\nmid 2D$, we have $\delta_p(n,f)=1-\frac{1}{p^{2\fd}}$. Note also that $\chi(p)=0$ for $p\mid n$. Then put \eqref{K-f-2} into the equality \eqref{K-f-1}, we see that both \eqref{KF} and \eqref{gamma2} follow.
\end{proof}

By the same method as in Corollary \ref{B} and Lemma \ref{delta}, we have

\begin{cor}\label{B2}
There is a finite set $\fS(d)$ of positive square free integers such that
\[\rho_f(n) \in \{\rho_f(s)\}_{s\in \fS(d)}\]
for all positive square free integer \(n\) coprime to \(2D\) .
\end{cor}

\bigskip

In particular, when \(\fd\) is even, which is equivalent to $d\geq5$ and $d\equiv 1\mod 4$,
the field \(K\) is a real quadratic field. The number $\fd$ is even.
Let \(w_{\fd}(K)\) be the largest integer \(N\) such that all the elements in \(Gal(K(\zeta_N)/N)\) can be annihilated by \(\fd\), where \(\zeta_N\) is a primitive \(N^{th}\) root of unity. Let \(K_{d-3}(\mathcal{O}_K)\) be the \((d-3)^{th}\) \(K\)-group of $\CO_K$ defined by Quillen, then we have the following Quillen-Lichtenbaum conjecture \cite{Lichtenbaum}:
\[|K_{d-3}(\mathcal{O}_K)|=\begin{cases}w_{\fd}(K)\zeta_K(1-\fd),\quad&\textrm{if } d\equiv 5\mod 8\\ \frac{w_\fd(K)}{4}\zeta_K(1-\fd),\quad&\textrm{if } d\equiv 1\mod 8.\end{cases}\]
Here $\zeta_K(s)$ is the Dedekind zeta function associated to $K$. By Wiles's theorem on Iwasawa main conjecture \cite{Wiles91} and $K/\BQ$ is abelian, we know the Quillen--Lichtenbaum conjecture is true for $K$ by the fundamental work of Rost and Voevodsky.

\medskip

\begin{proof}[Proof of Theorem \ref{main-thm-1}]
Since $d\geq5$ and $d\equiv 1\mod 4$, $\fd$ is even. Keeping the same notations as above, by Proposition \ref{K}, we have
\[G_f(n)=\rho_f(n)L(1-\fd,\chi).\]
From the above remark on Quillen--Lichtenbaum conjecture and the identity
\[L(1-\fd,\chi)\zeta(1-\fd)=\zeta_K(1-\fd),\]
we know the result follows.
\end{proof}

\bigskip

\section{Waldspurger's formula and applications to elliptic curves}\label{walds}

In this section, we keep the assumption that $d$ is odd. We first give some basic knowledge on Shimura lift and Waldspurger formula (see also \cite{Qin}). We then study the case  $d=3$ and prove Theorem \ref{main-thm-3}.

\subsection{Modular Forms}
Let \(M_k(N,\chi)\) be the space of modular forms of weight \(k\in\frac{1}{2}\mathbb{Z}\), level \(N\in\mathbb{N}\) and character \(\chi\), \(S_k(N,\chi)\) be the subspace of cusp forms, \(E_k(N,\chi)\) be the orthogonal complement of \(S_k(N,\chi)\) with respect to the Petersson inner product.

First we assume \(k\) is not an integer. For any \(h=\sum\limits_{n=0}^\infty a_nq^n\in M_k(N,\chi)\), the Hecke operator \(T_{p^2}\) with \(p\nmid N\) is given by \(T_{p^2}h=\sum\limits_{n=0}^\infty b_nq^n\) where
\begin{equation}\label{H2}
b_n=a_{p^2n}+\chi(p)\left(\frac{(-1)^{k-\frac{1}{2}}n}{p}\right)p^{k-\frac{3}{2}}a_n+\chi(p^2)p^{2k-2}a_{\frac{n}{p^2}}.
\end{equation}
Here \(a_{\frac{n}{p^2}}=0\) if \(p^2\nmid n\).

Note that when \(h\) is an eigenform with eigenvalue \(\lambda_p\) and \(a_1=1\), then
\begin{equation}\label{E2}
\lambda_p=a_{p^2}+\chi(p)\left(\frac{(-1)^{k-\frac{1}{2}}}{p}\right)p^{k-\frac{3}{2}}
\end{equation}

\medskip

Next, we consider the weight \(k\in\mathbb{Z}\). The Hecke operator \(T_{p}\) with \(p\nmid N\) is given by $T_p h=\sum^\infty_{n=0}b_nq^n$ where
\begin{equation}\label{H1}
b_n=a_{pn}+\chi(p)p^{k-1}a_{\frac{n}{p}}.
\end{equation}
Here \(a_{\frac{n}{p}}=0\) if \(p\nmid n\).
When \(h\) is an eigenform with eigenvalue \(\mu_p\) and \(a_1=1\), then
\begin{equation}\label{E1}
\mu_p=a_{p}.
\end{equation}

\bigskip

Recall that the theta series related to \(f\) is defined to be
\[\theta_f=\sum\limits_{x\in \mathbb{Z}^d} q^{f(x)}=\sum\limits_{n=0}^\infty r_f(n)q^n.\]
And the theta series related to the genus \(\fG_f\) of \(f\) is
\[\theta_{\fG_f}=\sum\limits_{n=0}^\infty G_f(n)q^n.\]

Recall that $h=h_f$ is the class number of $f$ and $N_f$ the level of $f$.
Write \(\fG_f=\{f_1=f,...,f_h\}\), \(\vartheta_i=\theta_f-\theta_{f_i}\) for \(i=2,...h\), and \(V_f=\mathbb{C}\theta_{f_1}\oplus \cdots \oplus \mathbb{C}\theta_{f_h}\). Let \(\chi_f=\left(\frac{4D}{.}\right)\), then we have the following decomposition:

\begin{lem}\label{Hecke}
\(V_f\) can be decomposed into
\[V_f\cap E_{\frac{d}{2}}(N_f,\chi_f) = \mathbb{C} \theta_{\fG_f}\]
\[V_f\cap S_{\frac{d}{2}}(N_f,\chi_f)=\mathbb{C}\vartheta_2\oplus \cdots \oplus \mathbb{C}\vartheta_h\]
which are preserved by Hecke operators.
\end{lem}

\begin{proof}
This is a well-known result, see also \cite{Qin}.
From \cite[Theorem 1.1]{Rallis}, we know \(V_f\) is preserved by Hecke operators.
By a result of Schulze-Pillot \cite[Korollar 1]{Schulze}, \(T_{p^2}\theta_{\fG_f}=(p^{d-2}+1)\theta_{\fG_f}\) and \(\theta_{\fG_f}\in E_{\frac{d}{2}}(N_f,\chi_f)\). Now a theorem of Eichler (see \cite[Theorem on Page 245]{Walling}) shows that $\vartheta_j(2\leq j\leq h)$ are cusp forms. Therefore the result follows.
\end{proof}

\medskip

We will consider the case when \(d=3\) and \(h_f=2\). The Lemma \ref{Hecke} implies that \(\vartheta_2\) is an eigenform. So we can apply the Shimura lift to \(\vartheta_2\). For this purpose, we give the following theorem:

\begin{thm}[Shimura \cite{Shimura}]\label{SL}
Let \(k\ge 3\) be an odd integer, \(\lambda=\frac{k-1}{2}\), \(4|N\), \(\chi\) be a Dirichlet character modulo \(N\). Let \(\theta\in S_{\frac{k}{2}}(N,\chi)\) be an eigenform for $T_{p^2}$ for all primes $p\nmid N$ with eigenvalue \(\lambda_p\). Define \(\Theta=\sum\limits_{n=1}^\infty b_nq^n\) by
\begin{equation}\label{SLF}
\sum\limits_{n=1}^\infty \frac{b_n}{n^s}=\prod\limits_p \frac{1}{1-\lambda_p p^{-s}+\chi(p^2)p^{k-2-2s}}
\end{equation}
Then \(\Theta\in M_{k-1}(N',\chi^2)\) for some integer \(N'\) which is divisible by the conductor of \(\chi^2\). If \(k\ge 5\), \(\Theta\) is a cusp form. Moreover, by a remark made by Niwa \cite{Niwa}, we can choose \(N'=\frac{N}{2}\).
\end{thm}

For a eigencusp new form $h$ with integral weight and a Dirichlet character $\eta$, we denote by $L(h,\eta,s)$ the $L$-functions associated to $h$ twisted by $\eta$. For a square free integer $n$, we denote \(\chi_n=(\frac{n}{\cdot})\).
Using Shimura lift, one can establish an relation between the coefficient of \(\theta\) and the values of the twisted $L$-functions of \(\Theta\).

\begin{thm}[Waldspurger \cite{Waldspurger}]\label{W}
Given \(\theta=\sum\limits_{n=1}^\infty a_nq^n\in S_{\frac{k}{2}}(N,\chi)\) satisfying \(T_{p^2}\theta=\lambda_p \theta\) for all \(p\nmid N\). If there is \(\Theta\in S_{k-1}(\frac{N}{2},\chi^2)\) such that \(T_p\Theta=\lambda_p \Theta\) for all \(p\nmid N\), then
\begin{equation}\label{WF}
a^2_{n_0}L(\Theta,\chi_{-1}^\lambda\chi^{-1}\chi_n,\lambda)n^{\frac{k}{2}-1}=a^2_{n}
L(\Theta,\chi_{-1}^\lambda\chi^{-1}\chi_{n_0},\lambda)n_0^{\frac{k}{2}-1}\chi\left(\frac{n_0}{n}\right)
\end{equation}
for any positive square free integers \(n,n_0\) with \(\frac{n_0}{n}\in (\mathbb{Q}_p^\times )^2\) for all \(p|N\).
\end{thm}

If we can find \(n_0\) with \(L(\Theta,\chi_{-1}^\lambda\chi^{-1}\chi_{n_0},\lambda)\neq 0\), then we can get an expression of \(a_n\) with respect to \(L(\Theta,\chi_{-1}^\lambda\chi^{-1}\chi_{n},\lambda)\). In particular, if \(\chi\) is a quadratic character, then
\begin{equation}\label{WF3}
a_{n}=\pm a_{n_0}\sqrt{\frac{L(\Theta,\chi_{-1}\chi_f^{-1}\chi_{n},1)n^{\frac{k}{2}-1}}{L(\Theta,\chi_{-1}\chi_f^{-1}\chi_{n_0},1)n_0^{\frac{k}{2}-1}}}.
\end{equation}
Here we recall $\chi_f=\left(\frac{4D}{.}\right)$.

\subsection{Application to $d=3$ and $h_f=2$}
Let \(d=3\) and $h_f=2$. In this case we write \(\fG_f=\{f,g\}\),  by Lemma \ref{Hecke}, \(\theta_f-\theta_g\) is an eigenform. We can give an explicit formula of \(r_f(n)\) applying both Theorems \ref{SL} and \ref{WF}. For this purpose, let us recall the following condition:
\begin{itemize}
  \item (\textbf{Sh-E}) The Shimura lift of $\theta_f-\theta_g$ corresponds to an elliptic curve $\CE$ over $\BQ$.
\end{itemize}
For any square free integer $m$, we set $\CE^{(m)}$ to be the quadratic twist of $\CE$ by the extension $\BQ(\sqrt{m})/\BQ$. If $m=1$, $\CE^{(m)}$ is equal to $\CE$. Let $L(\CE^{(m)},s)$ be the $L$-function of $\CE^{(m)}$. It is a holomorphic function on the whole complex plane. Let $\Omega_{\CE^{(m)}}$ be the real period of $\CE^{(m)}$ defined in the introduction. Recall that $D=b_1b_2b_3$ and $D^\circ$ is the square free part of $D$. We define the algebraic $L$-values of $\CE^{(-mD^\circ)}$ to be
\[\CL(m)=\frac{L(\CE^{(-mD^\circ)})}{\Omega_{\CE^{(-mD^\circ)}}}.\]
Let $\xi=\frac{\mu_f}{\mu_f+\mu_g}$ where $\mu_f=|\CO(A_f)|$. For any positive square free integer $b$, we denote by $h(-b)$ the class number of $\BQ(\sqrt{-b})$. Recall the rational number $\lambda_f(n)$ given in Proposition \ref{main-thm-2}.

\begin{prop}\label{MT}
Let
\[\theta_f-\theta_g=\sum^\infty_{n=1}a_nq^n.\]
Assume the condition (\textbf{Sh-E}) holds. For any positive square free integer $n$ coprime to $2D$, we have
\begin{equation}
r_f(n)=\xi a_{n}+\lambda_f(n)h(-nD^\circ).
\end{equation}

\end{prop}

\begin{proof}
Note that \(\chi_f=\left(\frac{4D}{.}\right)\). From Proposition \ref{C} and Theorem \ref{W}, we get a system of linear equations
	$$
    		\left\{
    		\begin{aligned}
        		& (1-\xi)r_f(n)+\xi r_g(n)=\lambda_f(n)h(-D^\circ n)\\
        		& r_f(n)-r_g(n)=a_n.
    		\end{aligned}
    		\right.
	$$
Since the determinant of the coefficients of the linear equation is non-zero, which is in fact equal to $-1$. The result follows by solving the linear equation via Cramer's rule.
\end{proof}

We recall the  {\it $N_f$-square-linked}  relation between
two square free positive integers defined in the introduction. We say the two such integers $n_1,n_2$ are {\it $N_f$-square-linked}  if $n_1/n_2$ is a square in $\BQ^\times_\ell$ for each prime $\ell$ dividing $N_f$.

Let $\fS_1$ be a set of positive square free integers which consists of representatives of the equivalence classes of the {\it $N_f$-square-linked}
relation. Note that $N_f$ and $2D$ have the same prime divisors. From the proof of Corollary \ref{gamma}, we can take $\fS_1$ as a subset of $\fS(3)$. we now make the following assumption on $\fS_1$:
\begin{itemize}
  \item (\textbf{In-Nv})There exists a positive square free integer $n_0\in\fS_1$ such that
  $$L(\CE^{(-D^\circ n_0)},1)\neq0.$$
\end{itemize}

\begin{prop}\label{MT2}
Keep the notation as in Proposition \ref{MT}. Under the condition (\textbf{In-Nv}), if $n$ and $n_0$ are
{\it $N_f$-square-linked}, then we have
\begin{equation}\label{EF}
r_f(n)=\pm \xi a_{n_0}\sqrt{\frac{\mathcal{L}(n)}{\mathcal{L}(n_0)}}+\lambda_f(n)h(-nD^\circ).
\end{equation}
\end{prop}

We remark from \eqref{WF3}, $\sqrt{\frac{\mathcal{L}(n)}{\mathcal{L}(n_0)}}$ is a rational number.

\begin{proof}
The result directly follows from Proposition \ref{MT} and \eqref{WF3}.
\end{proof}

From the formula \eqref{EF}, we know $r_f(n)$ is related to the algebraic $L$-values of elliptic curves. Now we can apply a $2$-divisibility result on algebraic $L$-values of the quadratic twists of elliptic curves by one of us, see \cite{Zhai}. We fix a normalised additive $2$-adic valuation $v_2$ on $\overline{\BQ}_2$ such that $v_2(2)=1$.

\medskip

\begin{thm}\label{Z}
Let \(E/\mathbb{Q}\) be an optimal eliptic curve with conductor \(C\). Then for all square free integer \(M\equiv 1\ (mod\ 4)\) with \((M,C)=1\), we have
\begin{equation}\label{ZF}
v_2(L(E^{(M)},1)/\Omega_{E^{(M)}})\ge t_E(M)-1-v_2(\nu_E),
\end{equation}
where \(\nu_E\) is the Manin constant of \(E\) and \(t_E(M)=\sum\limits_{q|M} t(q)\), with
$$t(q)=
    		\begin{cases}
        		1 & \textrm{if  $2|c_q(E^{(M)})$,} \\
        		0 & \textrm{if  $2\nmid c_q(\mathcal{E}^{(M)})$.}
    		\end{cases}$$
Here \(c_p(E)\) is the Tamagawa number of \(E/\mathbb{Q}\) at a finite prime \(p\).
\end{thm}

\medskip

Now, we can easily see the following theorem completes the proof of Theorem \ref{main-thm-3}.

\begin{thm}\label{V}
Keep the conditions of Propositions \ref{MT} and \ref{MT2}. Let \(n>3\) be a square free integer coprime to \(2D\). Define \(E=\CE^{(-D^\circ)}\) if \(n\equiv 1\ (mod\ 4)\) and \(E=\CE^{(D^\circ)}\) if \(n\equiv 3\ (mod\ 4)\). Assume that $E$ is optimal. Let \(n_0\in \fS_1\) satisfy (\textbf{In-Nv}). If $n$ and $n_0$ are {\it $N_f$-square-linked},
then
\begin{equation}\label{VF}
v_2(r_f(n))\ge \min \left\{A_{n_0}+\frac{1}{2}t_E\left((\frac{-1}{n})n\right),B_{n_1}+v_2(h(-nD^\circ))\right\},
\end{equation}
with \(n_1\in \fS(3)\) satisfying \(\lambda_f(n_1)=\lambda_f(n)\), where
	$$
    		\left\{
    		\begin{aligned}
        		& A_{n_0}=v_2\left(\frac{\xi a_{n_0}}{\sqrt{\mathcal{L}(n_0)}}\right)-\frac{1}{2}(1+v_2(\nu_\mathcal{E}))\\
        		& B_{n_1}=v_2(\lambda_f(n_1))
    		\end{aligned}
    		\right.
	$$
are constants depend only on \(n_0\in \fS_1\) satisfying (\textbf{In-Nv}) and \(n_1\in \fS(3)\).

Furthermore, let \(\mu(n)\) be the number of distinct odd prime divisors of \(n\). Then there exists constants \(\kappa \in \mathbb{N}\) and \(A\in \frac{1}{2}\mathbb{Z}\) such that
\begin{equation}\label{VF2}
v_2(r_f(n))\ge A+\frac{1}{2}t_E\left((\frac{-1}{n})n\right)
\end{equation}
holds for all $n$ satisfying \(\mu(n) \ge \kappa\).
\end{thm}

\begin{proof}
Equation (\ref{VF}) is a direct corollary of Propositions \ref{MT2} and \ref{Z}. By our choice, the set \(\fS_1\) is a finite subset of \(\fS(3)\), we can define $A$ to be the minimal value of $A_{n_0}$ such that $n_0\in \fS_1 \textrm{satisfying ((\textbf{In-Nv}))}$, and $B=\min \{B_{n_1}\}_{n_1\in \fS(3)}$.
Then equation (\ref{VF}) implies
\begin{equation}\label{VF3}
  v_2(r_f(n))\ge \min \left\{A+\frac{1}{2}t_E\left((\frac{-1}{n})n\right),B+v_2(h(-nD^\circ))\right\}.
\end{equation}
By Gauss's genus theory, we have \(v_2(h(-D^\circ n))\ge \mu(n)+\mu(D^\circ)-1\). Let \(\kappa=\max \{2(A-B+1-\mu(D^\circ)),1\}\), then for any $n$ satisfying \(\mu(n)\ge \kappa\), by using the inequalities
\[\frac{1}{2}\mu(n)\geq A-B+1-\mu(D^\circ), \]
and
\(t_E\left((\frac{-1}{n})n\right) \le \mu(n)\),
we get
\[B+v_2(h(-D^\circ n)) \ge A+\frac{1}{2}t_E\left((\frac{-1}{n})n\right).\]
Thus (\ref{VF2}) follows from \eqref{VF3}.
\end{proof}

\bigskip

\section{Examples and Numerical data for $v_2(r_f(n))$}\label{num-1}

\subsection{Examples} \label{Eg}
Let \(f(x,y,z)=x^2+y^2+pz^2\), where \(p>3\) is a prime number. In this case, \(N_f=4p\), \(D^\circ=p\). Consider \(n>3\), an integer square free and coprime to \(2p\).

\begin{prop}\label{gfp}
\(\lambda_f(n)\) has the following explicit form:
\begin{enumerate}
  \item If \(p\equiv 1\ (mod\ 4)\),
	$$
		\lambda_f(n)=
    		\left\{
    		\begin{aligned}
        		\frac{12}{p+1} & \quad pn\equiv 1\mod 8\\
        		\frac{24}{p+1} & \quad pn\equiv 3\mod 8\\
        		\frac{12}{p+1} & \quad pn\equiv 5\mod 8\\
        		0 & \quad pn\equiv 7\mod 8.
    		\end{aligned}
    		\right.
	$$
  \item If \(p\equiv 3\ (mod\ 4)\),
	$$
		\lambda_f(n)=
    		\left\{
    		\begin{aligned}
        		\frac{4}{p-1} & \quad pn\equiv 1\mod 8\\
        		\frac{24}{p-1} & \quad pn\equiv 3\mod 8\\
        		\frac{4}{p-1} & \quad pn\equiv 5\mod 8\\
        		\frac{16}{p-1} & \quad pn\equiv 7\mod 8.
    		\end{aligned}
    		\right.
	$$
\end{enumerate}
\end{prop}

\begin{proof}
It is a direct calculation of Lemma \ref{delta3}.
\end{proof}

So in this case we can take \(\fS_1=\fS(3)\).

\medskip

Now assume \(h_f=2\). This happens when \(p=7,11,13,17,29\).

\begin{cor}\label{E}
If \(p=7,11,13,17,29\), then the Shimura lift of \(\theta_f-\theta_g\) is a modular form corresponding to \(\CE=14a,11a,26b,34a,58b\). Moreover,
\begin{equation}\label{EF1}
r_f(n)=\pm \xi_fa_{n_0}\sqrt{\frac{\mathcal{L}(n)}{\mathcal{L}(n_0)}}+\lambda_f(n)h(-pn)
\end{equation}
with the same condition in Theorem \ref{W} if \(L(E^{(-pn_0)},1)\neq 0\), where

	$$
		\xi_f=
    		\left\{
    		\begin{aligned}
        		\frac{2}{3} & \quad p=7,17\\
        		\frac{4}{5} & \quad p=11,29\\
        		\frac{4}{7} & \quad p=13.
    		\end{aligned}
    		\right.
	$$
Furthermore, the corresponding \(\kappa=\kappa_p\) is given by
	$$
		\kappa_p=
    		\left\{
    		\begin{aligned}
        		1 & \quad p=17\\
        		2 & \quad p=13, 29\\
		3 & \quad p=7\\
        		4 & \quad p=11.
    		\end{aligned}
    		\right.
	$$
\end{cor}

\begin{proof}
Note that \(\mu_f=16\) for all \(p\) and
	$$
		\mu_g=
    		\left\{
    		\begin{aligned}
        		8 & \quad p=7,17\\
        		4 & \quad p=11,29\\
        		12 & \quad p=13.
    		\end{aligned}
    		\right.
	$$
the equation (\ref{EF}) follows from Theorem \ref{MT}.

For $p=7$, \(\CE=14a1\), \(A=3/2\), \(B=0\), \(\kappa=3\).
\begin{center}
\begin{scriptsize}
\begin{table}[H]
\center
\begin{tabular}{|c|c|c|c|c|}
    \hline \(n_0\)&\(\mathcal{L}(n_0)\)&\(a_{n_0}\)&\(A_{n_0}\)&\(B_{n_0}\)\\
    \hline 57 & 4 & -4 & 3/2 & 3\\
    \hline 43 & 4 & 4 & 3/2 & 0\\
    \hline 29 & 0 & 0 & / & 2\\
    \hline 15 & 4 & 4 & 3/2 & 0\\
    \hline 89 & 4 & -4 & 3/2 & 3\\
    \hline 19 & 4 & -4 & 3/2 & 0\\
    \hline 5 & 0 & 0 & / & 2\\
    \hline 103 & 16 & 8 & 3/2 & 0\\
    \hline
\end{tabular}
\end{table}
\end{scriptsize}
\end{center}

For $p=11$, \(\CE=11a1\), \(A=2\), \(B=0\), \(\kappa=4\).
\begin{center}
\begin{scriptsize}
\begin{table}[H]
\center
\begin{tabular}{|c|c|c|c|c|}
    \hline \(n_0\)&\(\mathcal{L}(n_0)\)&\(a_{n_0}\)&\(A_{n_0}\)&\(B_{n_0}\)\\
    \hline 89 & 90 & 6 & 2 & 2\\
    \hline 67 & 125 & -10 & 5/2 & 0\\
    \hline 133 & 40 & -4 & 2 & 3\\
    \hline 23 & 5 & -2 & 5/2 & 0\\
    \hline 57 & 160 & 8 & 2 & 2\\
    \hline 35 & 20 & -4 & 5/2 & 0\\
    \hline 101 & 40 & -4 & 2 & 3\\
    \hline 79 & 20 & 4 & 5/2 & 0\\
    \hline
\end{tabular}
\end{table}
\end{scriptsize}
\end{center}

For $p=13$, \(\CE=26b1\), \(A=2\), \(B=1\), \(\kappa=2\).
\begin{center}
\begin{scriptsize}
\begin{table}[H]
\center
\begin{tabular}{|c|c|c|c|c|}
    \hline \(n_0\)&\(\mathcal{L}(n_0)\)&\(a_{n_0}\)&\(A_{n_0}\)&\(B_{n_0}\)\\
    \hline 105 & 8 & 4 & 2 & 1\\
    \hline 131 & 0 & 0 & / & \(\infty\)\\
    \hline 53 & 8 & -4 & 2 & 1\\
    \hline 79 & 72 & -12 & 2 & 2\\
    \hline 41 & 8 & -4 & 2 & 1\\
    \hline 67 & 0 & 0 & / & \(\infty\)\\
    \hline 93 & 8 & -4 & 2 & 1\\
    \hline 15 & 2 & 2 & 2 & 2\\
    \hline
\end{tabular}
\end{table}
\end{scriptsize}
\end{center}

For $p=17$, \(\CE=34a1\), \(A=1\), \(B=1\), \(\kappa=1\).
\begin{center}
\begin{scriptsize}
\begin{table}[H]
\center
\begin{tabular}{|c|c|c|c|c|}
    \hline \(n_0\)&\(\mathcal{L}(n_0)\)&\(a_{n_0}\)&\(A_{n_0}\)&\(B_{n_0}\)\\
    \hline 137 & 200 & -20 & 1 & 1\\
    \hline 35 & 8 & 4 & 1 & 2\\
    \hline 69 & 72 & 12 & 1 & 1\\
    \hline 103 & 0 & 0 & / & \(\infty\)\\
    \hline 105 & 32 & -8 & 1 & 1\\
    \hline 139 & 8 & 4 & 1 & 2\\
    \hline 173 & 32 & 8 & 1 & 1\\
    \hline 71 & 0 & 0 & / & \(\infty\)\\
    \hline
\end{tabular}
\end{table}
\end{scriptsize}
\end{center}

For $p=29$, \(\CE=58b1\), \(A=2\), \(B=1\), \(\kappa=2\).
\begin{center}
\begin{scriptsize}
\begin{table}[H]
\center
\begin{tabular}{|c|c|c|c|c|}
    \hline \(n_0\)&\(\mathcal{L}(n_0)\)&\(a_{n_0}\)&\(A_{n_0}\)&\(B_{n_0}\)\\
    \hline 233 & 98 & 14 & 2 & 1\\
    \hline 59 & 0 & 0 & / & \(\infty\)\\
    \hline 349 & 18 & -6 & 2 & 1\\
    \hline 407 & 32 & 8 & 2 & 2\\
    \hline 89 & 32 & -8 & 2 & 1\\
    \hline 147 & 0 & 0 & / & \(\infty\)\\
    \hline 205 & 32 & 8 & 2 & 1\\
    \hline 31 & 8 & -4 & 2 & 2\\
    \hline
\end{tabular}
\end{table}
\end{scriptsize}
\end{center}
\end{proof}

In particular, if \(p\equiv 3\ (mod\ 4)\) and \(L(\CE^{(-pn)},1)=0\), then \(r_f(n)=r_g(n)>0\).

\medskip

In Corollary \ref{E}, we may take \(\CE:y^2=P(x)\), where
	$$
		P(x)=
    		\left\{
    		\begin{aligned}
        		& x^3+x^2-8x+16 & &\quad p=7\\
        		& x^3-4x^2+16 & &\quad p=11\\
		& x^3-3x^2-40x+208 & &\quad p=13\\
		& x^3+x^2-48x+64 & &\quad p=17\\
        		& x^3+5x^2+88x+592 & &\quad p=29.
    		\end{aligned}
    		\right.
	$$
Note that \(P(x)\) has an integral root when \(p=7\) or \(p=17\), in these cases \(t_E(n)=\mu(n)\).

\begin{cor}
If \(p=7\), then
\begin{equation}\label{num-f-1}
  v_2(r_f(n))\ge \frac{3}{2}+\frac{1}{2}\mu(n).
\end{equation}
If \(p=17\), then
\[v_2(r_f(n))\ge 1+\frac{1}{2}\mu(n).\]
\end{cor}

\begin{proof}
By Tate's algorithm \cite{Tate}, for any prime \(q|n\) we have
\[c_q(\CE^{(-pn)})=1+\#\{x\in \mathbb{F}_q:P(x)=0\}.\]
Since \(P(x)\) has a root in \(\mathbb{Z}\), \(c_q(\CE^{(-pn)})=2\) or \(4\). Finally note that \(v_2(r_f(n))\ge 3\) except \(n=1\).
\end{proof}

\medskip

\subsection{Data for $p=7$} Use the program in PARI/GP, we give the 2-valuation of \(r_f(n)\) in Section \ref{Eg} when \(p=7\). We can see that the choice of \(A\) is almost optimal.
\begin{verbatim}
for(n=5,1000,if(gcd(14,n)==1&&moebius(n)!=0&&omega(n)>2,s=0;
for(z=-floor(sqrt(n/7)),floor(sqrt(n/7)),for(y=-floor(sqrt(n-7*z^2)),
floor(sqrt(n-7*z^2)),if(issquare(n-7*z^2-y^2)==1,if(n-7*z^2-y^2==0,s+=1);
if(n-7*z^2-y^2>0,s+=2);)));print(n,", ",3/2+omega(n)/2,", ",s,);))
\end{verbatim}

\begin{center}
\begin{scriptsize}
\begin{table}[H]
\renewcommand{\arraystretch}{1.3}
\center
\begin{tabular}{|c|c|c|c|c|}
    \hline \(n\)&\(\mu(n)\)&\(\frac{3}{2}+\frac{1}{2}\mu(n)\)&\(r_f(n)\)&\(v_2(r_f(n))\)\\
\hline 165 & 3 & 3 & 32 & 5\\
\hline 195 & 3 & 3 & 16 & 4\\
\hline 255 & 3 & 3 & 16 & 4\\
\hline 285 & 3 & 3 & 32 & 5\\
\hline 345 & 3 & 3 & 96 & 5\\
\hline 429 & 3 & 3 & 32 & 5\\
\hline 435 & 3 & 3 & 32 & 5\\
\hline 465 & 3 & 3 & 96 & 5\\
\hline 555 & 3 & 3 & 32 & 5\\
\hline 561 & 3 & 3 & 112 & 4\\
\hline 615 & 3 & 3 & 16 & 4\\
\hline 627 & 3 & 3 & 16 & 4\\
\hline 645 & 3 & 3 & 64 & 6\\
\hline 663 & 3 & 3 & 48 & 4\\
\hline 2145 & 4 & 7/2 & 256 & 8\\
\hline 2805 & 4 & 7/2 & 128 & 7\\
\hline 3135 & 4 & 7/2 & 64 & 6\\
\hline 3315 & 4 & 7/2 & 96 & 5\\
\hline 3705 & 4 & 7/2 & 304 & 4\\
\hline 3795 & 4 & 7/2 & 64 & 6\\
\hline 4485 & 4 & 7/2 & 128 & 7\\
\hline 4785 & 4 & 7/2 & 288 & 5\\
\hline 4845 & 4 & 7/2 & 128 & 7\\
\hline 5115 & 4 & 7/2 & 96 & 5\\
\hline 5655 & 4 & 7/2 & 96 & 5\\
\hline 5865 & 4 & 7/2 & 384 & 7\\
\hline 6045 & 4 & 7/2 & 192 & 6\\
\hline 6105 & 4 & 7/2 & 288 & 5\\
\hline 6555 & 4 & 7/2 & 64 & 6\\
\hline 6765 & 4 & 7/2 & 192 & 6\\
    \hline
\end{tabular}
\end{table}
\end{scriptsize}
\end{center}

\begin{center}
\begin{scriptsize}
\begin{table}[H]
\renewcommand{\arraystretch}{1.3}
\center
\begin{tabular}{|c|c|c|c|c|}
    \hline \(n\)&\(\mu(n)\)&\(\frac{3}{2}+\frac{1}{2}\mu(n)\)&\(r_f(n)\)&\(v_2(r_f(n))\)\\
\hline 36465 & 5 & 4 & 672 & 5\\
\hline 40755 & 5 & 4 & 224 & 5\\
\hline 49335 & 5 & 4 & 192 & 6\\
\hline 53295 & 5 & 4 & 288 & 5\\
\hline 62205 & 5 & 4 & 512 & 9\\
\hline 62985 & 5 & 4 & 1120 & 5\\
\hline 64515 & 5 & 4 & 256 & 8\\
\hline 66495 & 5 & 4 & 288 & 5\\
\hline 72105 & 5 & 4 & 1088 & 6\\
\hline 76245 & 5 & 4 & 640 & 7\\
\hline 79365 & 5 & 4 & 640 & 7\\
\hline 81345 & 5 & 4 & 1024 & 10\\
\hline 85215 & 5 & 4 & 352 & 5\\
\hline 86955 & 5 & 4 & 320 & 6\\
\hline 87945 & 5 & 4 & 1568 & 5\\
\hline 692835 & 6 & 9/2 & 1024 & 10\\
\hline 838695 & 6 & 9/2 & 960 & 6\\
\hline 937365 & 6 & 9/2 & 2048 & 11\\
\hline 1057485 & 6 & 9/2 & 2560 & 9\\
\hline 1130415 & 6 & 9/2 & 1344 & 6\\
\hline 1181895 & 6 & 9/2 & 1152 & 7\\
\hline 1225785 & 6 & 9/2 & 5504 & 7\\
\hline 1263405 & 6 & 9/2 & 2304 & 8\\
\hline 1349205 & 6 & 9/2 & 3072 & 10\\
\hline 1430715 & 6 & 9/2 & 1536 & 9\\
\hline 1448655 & 6 & 9/2 & 1088 & 6\\
\hline 1495065 & 6 & 9/2 & 4864 & 8\\
    \hline
\end{tabular}
\end{table}
\end{scriptsize}
\end{center}

We remark the inequality \eqref{num-f-1} is equivalent to
\[  v_2(r_f(n))\ge \left\lceil\frac{3}{2}+\frac{1}{2}\mu(n)\right\rceil.
\]
If $n=3705$, the above data show the equality holds, i.e., our lower bound is in some sense optimal. Here for a real number $b$, we write $\lceil b\rceil$ for the least integer no smaller than $b$.

\subsection{Small Level Forms} For \(d=3\), \(h_f=2\) and \(N_f\le 100\), we give the Shimura lift of \(\theta_f-\theta_g\) with data in \cite{Brandt}. Denote \(g=a_1x^2+a_2y^2+a_3z^2+a_4yz+a_5xz+a_6xy\) by \([a_1,a_2,a_3,a_4,a_5,a_6]\). For many genus in the table, we can give eplicit form of Theorem \ref{MT} and \ref{V}.

\begin{center}
\begin{scriptsize}
\begin{table}[H]
\renewcommand{\arraystretch}{1.3}
\center
\begin{tabular}{|c|c|c|c|}
    \hline Level&\(f\)&\(g\)&Elliptic Curve\\
\hline 28&\([1,1,7,0,0,0]\)&\([1,2,4,2,0,0]\)&14a\\
\hline 40&\([1,1,10,0,0,0]\)&\([2,2,3,0,2,0]\)&20a\\
\hline 44&\([1,1,11,0,0,0]\)&\([1,3,4,2,0,0]\)&11a\\
\hline 52&\([1,1,13,0,0,0]\)&\([2,2,5,2,2,2]\)&26b\\
\hline 52&\([1,2,7,2,0,0]\)&\([2,3,3,2,0,2]\)&26a\\
\hline 56&\([1,1,14,0,0,0]\)&\([1,2,7,0,0,0]\)&14a\\
\hline 60&\([1,2,8,2,0,0]\)&\([1,3,5,0,0,0]\)&15a\\
\hline 60&\([1,1,15,0,0,0]\)&\([1,4,4,2,0,0]\)&30a\\
\hline 64&\([1,1,16,0,0,0]\)&\([2,2,5,2,2,0]\)&/ \\
\hline 68&\([1,3,6,2,0,0]\)&\([2,3,4,2,2,2]\)&17a\\
\hline 68&\([1,1,17,0,0,0]\)&\([1,2,9,2,0,0]\)&34a\\
    \hline
\end{tabular}
\end{table}
\end{scriptsize}
\end{center}

\begin{center}
\begin{scriptsize}
\begin{table}[H]
\renewcommand{\arraystretch}{1.3}
\center
\begin{tabular}{|c|c|c|c|}
    \hline Level&\(f\)&\(g\)&Elliptic Curve\\
\hline 72&\([1,2,9,0,0,0]\)&\([2,3,4,2,0,2]\)&36a\\
\hline 72&\([1,1,18,0,0,0]\)&\([2,2,5,0,2,0]\)&36a\\
\hline 76&\([1,2,10,2,0,0]\)&\([2,2,7,2,2,2]\)&38b\\
\hline 80&\([1,4,6,4,0,0]\)&\([2,2,5,0,0,0]\)&20a\\
\hline 80&\([1,1,20,0,0,0]\)&\([1,4,5,0,0,0]\)&20a\\
\hline 84&\([1,3,7,0,0,0]\)&\([2,3,4,0,2,0]\)&21a\\
\hline 84&\([1,2,11,2,0,0]\)&\([3,3,3,0,2,2]\)&14a\\
\hline 84&\([1,5,5,4,0,0]\)&\([2,2,7,0,0,2]\)&42a\\
\hline 88&\([1,2,11,0,0,0]\)&\([2,3,4,2,0,0]\)&44a\\
\hline 88&\([1,1,22,0,0,0]\)&\([2,3,5,2,2,2]\)&11a\\
\hline 96&\([1,2,12,0,0,0]\)&\([2,3,4,0,0,0]\)&24a\\
\hline 96&\([1,3,8,0,0,0]\)&\([3,3,4,-2,2,2]\)&24a\\
\hline 100&\([1,1,25,0,0,0]\)&\([2,3,5,2,2,0]\)&50b\\
\hline 100&\([1,2,13,2,0,0]\)&\([2,2,9,2,2,2]\)&50b\\
    \hline
\end{tabular}
\end{table}
\end{scriptsize}
\end{center}

\section{Numerical Data for $K$-groups}\label{ND}

We consider the standard form \(f=x_1^2+\cdots +x_d^2\) whose coefficients are 1. In this case \(K=\mathbb{Q}(\sqrt{n})\). We give the first three examples \(d=5,9,13\).

\subsection{Case $d=5$}

If we consider the simplest form \(f=x_1^2+x_2^2+x_3^2+x_4^2+x_5^2\) whose class number is 1, let \(K=\mathbb{Q}(\sqrt{n})\), then for odd number \(n>5\)
	$$
		r_f(n)=
    		\left\{
    		\begin{aligned}
        		& 20|K_2(\mathcal{O}_K)| & n\equiv 3\ (mod\ 4) \\
        		& 60|K_2(\mathcal{O}_K)| & n\equiv 1\ (mod\ 8) \\
			& 140|K_2(\mathcal{O}_K)| & n\equiv 5\ (mod\ 8)
    		\end{aligned}
    		\right.
	$$

\begin{center}
\begin{scriptsize}
\begin{table}[H]
\renewcommand{\arraystretch}{1.3}
\center
\begin{tabular}{|c|c|c|c|c|}
    \hline \(n\)&\(\tau(n)\)&\(r_f(n)\)&\(|K_2(\mathcal{O}_K)|\)&Factorization\\
\hline 7 & 1 & 320 & 16 & \(2^4\)\\
\hline 11 & 1 & 560 & 28 & \(2^2\cdot7\)\\
\hline 13 & 1 & 560 & 4 & \(2^2\)\\
\hline 15 & 2 & 960 & 48 & \(2^4\cdot3\)\\
\hline 17 & 1 & 480 & 8 & \(2^3\)\\
\hline 19 & 1 & 1520 & 76 & \(2^2\cdot19\)\\
\hline 21 & 2 & 1120 & 8 & \(2^3\)\\
\hline 23 & 1 & 1600 & 80 & \(2^4\cdot5\)\\
\hline 29 & 1 & 1680 & 12 & \(2^2\cdot3\)\\
\hline 31 & 1 & 3200 & 160 & \(2^5\cdot5\)\\
\hline 33 & 2 & 1440 & 24 & \(2^3\cdot3\)\\
\hline 35 & 2 & 3040 & 152 & \(2^3\cdot19\)\\
\hline 37 & 1 & 2800 & 20 & \(2^2\cdot5\)\\
\hline 39 & 2 & 4160 & 208 & \(2^4\cdot13\)\\
\hline 41 & 1 & 1920 & 32 & \(2^5\)\\
\hline 43 & 1 & 5040 & 252 & \(2^2\cdot3^2\cdot7\)\\
\hline 47 & 1 & 4480 & 224 & \(2^5\cdot7\)\\
\hline 51 & 2 & 6240 & 312 & \(2^3\cdot3\cdot13\)\\
\hline 53 & 1 & 3920 & 28 & \(2^2\cdot7\)\\
\hline 55 & 2 & 7360 & 368 & \(2^4\cdot23\)\\
\hline 57 & 2 & 3360 & 56 & \(2^3\cdot7\)\\
\hline 59 & 1 & 6800 & 340 & \(2^2\cdot5\cdot17\)\\
\hline 61 & 1 & 6160 & 44 & \(2^2\cdot11\)\\
\hline 65 & 2 & 3840 & 64 & \(2^6\)\\
\hline 67 & 1 & 9840 & 492 & \(2^2\cdot3\cdot41\)\\
\hline 69 & 2 & 6720 & 48 & \(2^4\cdot3\)\\
    \hline
\end{tabular}
\end{table}
\end{scriptsize}
\end{center}

\begin{center}
\begin{scriptsize}
\begin{table}[H]
\renewcommand{\arraystretch}{1.3}
\center
\begin{tabular}{|c|c|c|c|c|}
    \hline \(n\)&\(\tau(n)\)&\(r_f(n)\)&\(|K_2(\mathcal{O}_K)|\)&Factorization\\
\hline 71 & 1 & 9280 & 464 & \(2^4\cdot29\)\\
\hline 73 & 1 & 5280 & 88 & \(2^3\cdot11\)\\
\hline 77 & 2 & 6720 & 48 & \(2^4\cdot3\)\\
\hline 79 & 1 & 13440 & 672 & \(2^5\cdot3\cdot7\)\\
\hline 83 & 1 & 10320 & 516 & \(2^2\cdot3\cdot43\)\\
\hline 85 & 2 & 10080 & 72 & \(2^3\cdot3^2\)\\
\hline 87 & 2 & 12480 & 624 & \(2^4\cdot3\cdot13\)\\
\hline 89 & 1 & 6240 & 104 & \(2^3\cdot13\)\\
\hline 91 & 2 & 16480 & 824 & \(2^3\cdot103\)\\
\hline 93 & 2 & 10080 & 72 & \(2^3\cdot3^2\)\\
\hline 95 & 2 & 13760 & 688 & \(2^4\cdot43\)\\
\hline 97 & 1 & 8160 & 136 & \(2^3\cdot17\)\\
    \hline
\end{tabular}
\end{table}
\end{scriptsize}
\end{center}

\subsection{Case $d=9$}
For \(f=x_1^2+x_2^2+\cdots+x_9^2\) whose class number is 2, the other form in the genus of \(f\) has matrix

\[A_g=\begin{pmatrix}
2 & 0 & 0 & 0 & 0 & 0 & 0 & 0 & 0\\
0 & 4 & 2 & 2 & 2 & 2 & 2 & 2 & 2\\
0 & 2 & 4 & 2 & 2 & 2 & 2 & 2 & 2\\
0 & 2 & 2 & 4 & 2 & 0 & 0 & 0 & 0\\
0 & 2 & 2 & 2 & 4 & 0 & 2 & 2 & 2\\
0 & 2 & 2 & 0 & 0 & 4 & 2 & 2 & 2\\
0 & 2 & 2 & 0 & 2 & 2 & 4 & 2 & 2\\
0 & 2 & 2 & 0 & 2 & 2 & 2 & 4 & 2\\
0 & 2 & 2 & 0 & 2 & 2 & 2 & 2 & 4\\
\end{pmatrix}\]

Let \(K=\mathbb{Q}(\sqrt{n})\), then for odd number \(n>5\)

	$$
		15r_f(n)+2r_g(n)=
    		\left\{
    		\begin{aligned}
        		& 480|K_6(\mathcal{O}_K)| & n\equiv 3\ (mod\ 4) \\
        		& 65760|K_6(\mathcal{O}_K)| & n\equiv 1\ (mod\ 8) \\
			& 73440|K_6(\mathcal{O}_K)| & n\equiv 5\ (mod\ 8)
    		\end{aligned}
    		\right.
	$$

\begin{center}
\begin{scriptsize}
\begin{table}[H]
\renewcommand{\arraystretch}{1.3}
\center
\begin{tabular}{|c|c|c|c|c|c|}
    \hline \(n\)&\(\tau(n)\)&\(r_f(n)\)&\(r_g(n)\)&\(|K_6(\mathcal{O}_K)|\)&Factorization\\
\hline 7 & 1 & 12672 & 13440 & 452 & \(2^2\cdot113\)\\
\hline 11 & 1 & 60768 & 60960 & 2153 & \(2153\)\\
\hline 13 & 1 & 125280 & 125280 & 29 & \(29\)\\
\hline 15 & 2 & 182400 & 178560 & 6444 & \(2^2\cdot3^2\cdot179\)\\
\hline 17 & 1 & 317376 & 315840 & 82 & \(2\cdot41\)\\
\hline 19 & 1 & 421344 & 423840 & 14933 & \(109\cdot137\)\\
\hline 21 & 2 & 665280 & 665280 & 154 & \(2\cdot7\cdot11\)\\
\hline 23 & 1 & 800640 & 804480 & 28372 & \(2^2\cdot41\cdot173\)\\
\hline 29 & 1 & 2034720 & 2034720 & 471 & \(3\cdot157\)\\
\hline 31 & 1 & 2338560 & 2346240 & 82856 & \(2^3\cdot10357\)\\
\hline 33 & 2 & 3273792 & 3263040 & 846 & \(2\cdot3^2\cdot47\)\\
\hline 35 & 2 & 3487680 & 3474240 & 123466 & \(2\cdot7\cdot8819\)\\
\hline 37 & 1 & 4877280 & 4877280 & 1129 & \(1129\)\\
\hline 39 & 2 & 5165952 & 5151360 & 182900 & \(2^2\cdot5^2\cdot31\cdot59\)\\
\hline 41 & 1 & 6930432 & 6942720 & 1792 & \(2^8\cdot7\)\\
\hline 43 & 1 & 7334496 & 7345440 & 259809 & \(3\cdot11\cdot7873\)\\
\hline 47 & 1 & 9762048 & 9772800 & 345784 & \(2^3\cdot43223\)\\
    \hline
\end{tabular}
\end{table}
\end{scriptsize}
\end{center}

\begin{center}
\begin{scriptsize}
\begin{table}[H]
\renewcommand{\arraystretch}{1.3}
\center
\begin{tabular}{|c|c|c|c|c|c|}
    \hline \(n\)&\(\tau(n)\)&\(r_f(n)\)&\(r_g(n)\)&\(|K_6(\mathcal{O}_K)|\)&Factorization\\
\hline 51 & 2 & 13202112 & 13239360 & 467730 & \(2\cdot3^2\cdot5\cdot5197\)\\
\hline 53 & 1 & 16740000 & 16740000 & 3875 & \(5^3\cdot31\)\\
\hline 55 & 2 & 17377920 & 17358720 & 615388 & \(2^2\cdot23\cdot6689\)\\
\hline 57 & 2 & 22178112 & 22198080 & 5734 & \(2\cdot47\cdot61\)\\
\hline 59 & 1 & 21710880 & 21686880 & 768827 & \(271\cdot2837\)\\
\hline 61 & 1 & 28127520 & 28127520 & 6511 & \(17\cdot383\)\\
\hline 65 & 2 & 34755840 & 34725120 & 8984 & \(2^3\cdot1123\)\\
\hline 67 & 1 & 34647264 & 34596000 & 1226877 & \(3\cdot408959\)\\
\hline 69 & 2 & 42768000 & 42768000 & 9900 & \(2^2\cdot3^2\cdot5^2\cdot11\)\\
\hline 71 & 1 & 41526144 & 41585280 & 1470964 & \(2^2\cdot11\cdot101\cdot331\)\\
\hline 73 & 1 & 53342784 & 53344320 & 13790 & \(2\cdot5\cdot7\cdot197\)\\
\hline 77 & 2 & 61845120 & 61845120 & 14316 & \(2^2\cdot3\cdot1193\)\\
\hline 79 & 1 & 61855488 & 61850880 & 2190696 & \(2^3\cdot3\cdot37\cdot2467\)\\
\hline 83 & 1 & 71441568 & 71462880 & 2530311 & \(3\cdot7^3\cdot2459\)\\
\hline 85 & 2 & 89760960 & 89760960 & 20778 & \(2\cdot3\cdot3463\)\\
\hline 87 & 2 & 85281408 & 85311360 & 3020508 & \(2^2\cdot3^2\cdot83903\)\\
\hline 89 & 1 & 104481216 & 104479680 & 27010 & \(2\cdot5\cdot37\cdot73\)\\
\hline 91 & 2 & 101433024 & 101451840 & 3592498 & \(2\cdot7\cdot13\cdot19739\)\\
\hline 93 & 2 & 121279680 & 121279680 & 28074 & \(2\cdot3\cdot4679\)\\
\hline 95 & 2 & 114906240 & 114810240 & 4069196 & \(2^2\cdot1017299\)\\
\hline 97 & 1 & 144280512 & 144254400 & 37298 & \(2\cdot17\cdot1097\)\\
    \hline
\end{tabular}
\end{table}
\end{scriptsize}
\end{center}

\subsection{Case $d=13$}

For \(f=x_1^2+x_2^2+\cdots+x_{13}^2\) whose class number is 3, the other forms in the genus of \(f\) has matrix
\[A_g=
\left(\begin{array}{ccccccccccccc}
2 & 0 & 0 & 0 & 0 & 0 & 0 & 0 & 0 & 0 & 0 & 0 & 0\\
0 & 2 & 0 & 0 & 0 & 0 & 0 & 0 & 0 & 0 & 0 & 0 & 0\\
0 & 0 & 2 & 0 & 0 & 0 & 0 & 0 & 0 & 0 & 0 & 0 & 0\\
0 & 0 & 0 & 2 & 0 & 0 & 0 & 0 & 0 & 0 & 0 & 0 & 0\\
0 & 0 & 0 & 0 & 2 & 0 & 0 & 0 & 0 & 0 & 0 & 0 & 0\\
0 & 0 & 0 & 0 & 0 & 4 & 2 & 2 & 2 & 2 & 2 & 2 & 2\\
0 & 0 & 0 & 0 & 0 & 2 & 4 & 2 & 2 & 0 & 0 & 0 & 0\\
0 & 0 & 0 & 0 & 0 & 2 & 2 & 4 & 2 & 2 & 2 & 2 & 2\\
0 & 0 & 0 & 0 & 0 & 2 & 2 & 2 & 4 & 0 & 2 & 2 & 2\\
0 & 0 & 0 & 0 & 0 & 2 & 0 & 2 & 0 & 4 & 2 & 2 & 2\\
0 & 0 & 0 & 0 & 0 & 2 & 0 & 2 & 2 & 2 & 4 & 2 & 2\\
0 & 0 & 0 & 0 & 0 & 2 & 0 & 2 & 2 & 2 & 2 & 4 & 2\\
0 & 0 & 0 & 0 & 0 & 2 & 0 & 2 & 2 & 2 & 2 & 2 & 4\\
\end{array}\right)\]

\[A_h=
\left(\begin{array}{ccccccccccccc}
2 & 0 & 0 & 0 & 0 & 0 & 0 & 0 & 0 & 0 & 0 & 0 & 0\\
0 & 4 & 2 & 2 & 2 & -2 & 2 & 2 & 2 & 2 & 0 & 0 & 2\\
0 & 2 & 4 & 2 & 2 & 0 & 2 & 2 & 2 & 2 & 0 & 0 & 0\\
0 & 2 & 2 & 4 & 2 & -2 & 2 & 2 & 2 & 2 & 2 & 2 & 2\\
0 & 2 & 2 & 2 & 4 & -2 & 2 & 2 & 2 & 2 & 2 & 2 & 2\\
0 & -2 & 0 & -2 & -2 & 4 & -2 & -2 & -2 & -2 & -2 & -2 & -2\\
0 & 2 & 2 & 2 & 2 & -2 & 4 & 2 & 2 & 2 & 2 & 2 & 2\\
0 & 2 & 2 & 2 & 2 & -2 & 2 & 4 & 2 & 2 & 2 & 2 & 2\\
0 & 2 & 2 & 2 & 2 & -2 & 2 & 2 & 4 & 2 & 2 & 2 & 2\\
0 & 2 & 2 & 2 & 2 & -2 & 2 & 2 & 2 & 4 & 2 & 2 & 2\\
0 & 0 & 0 & 2 & 2 & -2 & 2 & 2 & 2 & 2 & 6 & 4 & 4\\
0 & 0 & 0 & 2 & 2 & -2 & 2 & 2 & 2 & 2 & 4 & 6 & 4\\
0 & 2 & 0 & 2 & 2 & -2 & 2 & 2 & 2 & 2 & 4 & 4 & 6\\
\end{array}\right)\]

Let \(K=\mathbb{Q}(\sqrt{n})\), then for odd number \(n>13\), then
	$$
		15r_f(n)+286r_g(n)+390r_h(n)=
    		\left\{
    		\begin{aligned}
        		& 260|K_{10}(\mathcal{O}_K)| & n\equiv 3\ (mod\ 4) \\
        		& 507780|K_{10}(\mathcal{O}_K)| & n\equiv 1\ (mod\ 8) \\
			& 524420|K_{10}(\mathcal{O}_K)| & n\equiv 5\ (mod\ 8).
    		\end{aligned}
    		\right.
	$$

\begin{center}
\begin{scriptsize}
\begin{table}[H]
\renewcommand{\arraystretch}{1.3}
\center
\begin{tabular}{|c|c|c|c|c|c|}
    \hline \(n\)&\(\tau(n)\)&\(r_f(n)\)&\(r_g(n)\)&\(r_h(n)\)&\(|K_{10}(\mathcal{O}_K)|\)\\
\hline 15 & 2 & 17689152 & 17682240 & 17678784 & \(2^4\cdot3\cdot661\cdot1481\)\\
\hline 17 & 1 & 34039200 & 34043040 & 34044960 & \(2^3\cdot5791\)\\
\hline 19 & 1 & 64966096 & 64973840 & 64977712 & \(2^2\cdot29\cdot1488671\)\\
\hline 21 & 2 & 109157152 & 109119520 & 109100704 & \(2^3\cdot17971\)\\
\hline 23 & 1 & 185245632 & 185287360 & 185308224 & \(2^4\cdot5\cdot7^3\cdot131\cdot137\)\\
\hline 29 & 1 & 643027632 & 643057200 & 643071984 & \(2^2\cdot3^2\cdot23537\)\\
\hline 31 & 1 & 959840128 & 959661440 & 959572096 & \(2^5\cdot5\cdot15939757\)\\
\hline 33 & 2 & 1309033440 & 1309137120 & 1309188960 & \(2^3\cdot3\cdot74231\)\\
\hline 35 & 2 & 1865484192 & 1865397280 & 1865353824 & \(2^3\cdot359\cdot383\cdot4507\)\\
\hline 37 & 1 & 2461908176 & 2462086480 & 2462175632 & \(2^2\cdot5\cdot7\cdot23173\)\\
\hline 39 & 2 & 3387200960 & 3387404480 & 3387506240 & \(2^4\cdot7\cdot80382319\)\\
\hline 41 & 1 & 4314910080 & 4314648960 & 4314518400 & \(2^5\cdot17\cdot43\cdot251\)\\
\hline 43 & 1 & 5802894864 & 5802799440 & 5802751728 & \(2^2\cdot3\cdot1285165781\)\\
\hline 47 & 1 & 9438089856 & 9438344320 & 9438471552 & \(2^5\cdot347\cdot2259041\)\\
\hline 51 & 2 & 14812980000 & 14813308320 & 14813472480 & \(2^3\cdot3\cdot1640393453\)\\
\hline 53 & 1 & 17722663920 & 17722632560 & 17722616880 & \(2^2\cdot983\cdot5939\)\\
\hline 55 & 2 & 22469066048 & 22468640320 & 22468427456 & \(2^4\cdot43\cdot139\cdot624419\)\\
\hline 57 & 2 & 26455010400 & 26453962080 & 26453437920 & \(2^3\cdot1327\cdot3391\)\\
\hline 59 & 1 & 32969103024 & 32968811120 & 32968665168 & \(2^2\cdot5\cdot401\cdot881\cdot12401\)\\
\hline 61 & 1 & 38507828528 & 38508883120 & 38509410416 & \(2^2\cdot13\cdot975797\)\\
\hline 65 & 2 & 54401717760 & 54403499520 & 54404390400 & \(2^7\cdot43\cdot13451\)\\
\hline 67 & 1 & 66523533648 & 66522998160 & 66522730416 & \(2^2\cdot3^2\cdot29\cdot139\cdot347\cdot3511\)\\
\hline 69 & 2 & 75740433600 & 75739949760 & 75739707840 & \(2^4\cdot3\cdot23\cdot90397\)\\
\hline 71 & 1 & 91274369472 & 91273643200 & 91273280064 & \(2^4\cdot79\cdot191911991\)\\
\hline 73 & 1 & 103288355040 & 103286031840 & 103284870240 & \(2^3\cdot7\cdot1087\cdot2309\)\\
\hline 77 & 2 & 138263995584 & 138263167680 & 138262753728 & \(2^4\cdot3\cdot11\cdot345041\)\\
\hline 79 & 1 & 164649144192 & 164650423680 & 164651063424 & \(2^5\cdot3^3\cdot17\cdot181\cdot164599\)\\
\hline 83 & 1 & 215420929776 & 215421624240 & 215421971472 & \(2^2\cdot3\cdot47710406203\)\\
\hline 85 & 2 & 238800815136 & 238799788320 & 238799274912 & \(2^3\cdot3\cdot17\cdot771209\)\\
\hline 87 & 2 & 279455322432 & 279454900800 & 279454689984 & \(2^4\cdot3\cdot6359\cdot2433247\)\\
\hline 89 & 1 & 306394252320 & 306389970720 & 306387829920 & \(2^3\cdot7^2\cdot1063627\)\\
\hline 91 & 2 & 358378620704 & 358376659360 & 358375678688 & \(2^3\cdot61\cdot16447\cdot118669\)\\
\hline 93 & 2 & 391074706464 & 391076432160 & 391077295008 & \(2^3\cdot3\cdot1289\cdot16657\)\\
\hline 95 & 2 & 452766709824 & 452767380800 & 452767716288 & \(2^4\cdot127\cdot592183489\)\\
\hline 97 & 1 & 493176047520 & 493181765280 & 493184624160 & \(2^3\cdot83892047\)\\
    \hline
\end{tabular}
\end{table}
\end{scriptsize}
\end{center}

\bigskip

\noindent Li-Tong Deng \\
Yau Mathematical Sciences Center \\
Tsinghua University \\
Beijing, 100084 \\
China \\
{\it dlt19@tsinghua.org.cn}

\bigskip

\noindent Yong-Xiong Li \\
Yau Mathematical Sciences Center \\
Tsinghua University \\
Beijing, 100084 \\
China \\
{\it yongxiongli@gmail.com}\\

\bigskip

\noindent Shuai Zhai\\
Research Center for Mathematics and Interdisciplinary Sciences \\
Shandong University \\
Qingdao, Shandong \\
China \\
{\it zhai@sdu.edu.cn}\\

\end{document}